\documentclass[11
pt]{amsart}
\usepackage{amssymb,amsmath,epsfig,mathrsfs, enumerate, xparse, mathtools}
\usepackage[pagewise]{lineno}
\usepackage[nodisplayskipstretch]{setspace}
\usepackage{graphicx}\usepackage[normalem]{ulem}
\usepackage{fancyhdr}
\usepackage[margin=2.5cm]{geometry}
\usepackage{hyperref}
\hypersetup{
colorlinks   = true,
urlcolor     = blue,
linkcolor    = blue,
citecolor   = red ,
bookmarksopen=true
}

\newtheorem{theorem}{Theorem}[section]
\theoremstyle{plain}

\newtheorem{corollary}[theorem]{Corollary}

\newtheorem{definition}[theorem]{Definition}

\newtheorem{lemma}[theorem]{Lemma}

\newtheorem{proposition}[theorem]{Proposition}
\newtheorem{remark}[theorem]{Remark}

\numberwithin{equation}{section}
\newcommand{\R}{\mathbb{R}}
\newcommand{\pa}{\partial}

\newcommand{\Om}{\Omega}

\newcommand{\Hn}{\mathbb{H}^{n}}
\newcommand{\lm}{\lambda}

\newcommand{\g}{\mathfrak{g}}
\newcommand{\Tau}{\Gamma}
\newcommand{\G}{\mathbb G}
\newcommand{\Ls}{\mathscr{L}}
\newcommand{\ve}{\varepsilon}

\newcommand{\vf}{\varphi}

\newcommand{\V}{\mathscr V}
\newcommand{\s}{\sigma}
\newcommand{\U}{\mathscr U}

\newcommand{\Scr}{\mathscr{S}}

\title[Schauder Estimates]{Higher order  Boundary Schauder Estimates in Carnot Groups}

\author{Agnid Banerjee}
\address{Arizona State University \\ Tempe, USA (former address: Tata Institute for Fundamental Research CAM, Bangalore-560065, INDIA)}
\email[Agnid Banerjee]{agnid.banerjee@asu.edu}

\author{Nicola Garofalo}
\address{Dipartimento di Ingegneria Civile, Edile e Ambientale (DICEA) \\ Universit\`a di Padova\\ 35131 Padova, ITALY}
\email[Nicola Garofalo]{nicola.garofalo@unipd.it}

\thanks{A.B was supported in part by  Department of Atomic Energy,  Government of India, under
project no.  12-R \& D-TFR-5.01-0520. N.G. was supported in part by a Progetto SID: ``Non-local Sobolev and isoperimetric inequalities", Univ. of Padova, 2019, and by a BIRD grant: ``Aspects of nonlocal operators via fine properties of heat kernels", Univ. of Padova, 2022. He has also been partially supported by a Visiting Professorship at the Arizona State University.}

\author{Isidro H. Munive}
\address{Universidad de Guadalajara}\email[Isidro H. Munive]{isidro.munive@academicos.udg.mx}
\begin{document}

\begin{abstract}
In his seminal 1981 study D. Jerison showed the remarkable negative phenomenon that there exist, in general, no Schauder estimates near the characteristic boundary in the Heisenberg group $\Hn$. On the positive side, by adapting tools from Fourier and microlocal analysis, he developed a Schauder theory at a non-characteristic portion of the boundary, based on the non-isotropic Folland-Stein H\"older classes. On the other hand, the 1976 celebrated work of Rothschild and Stein on their lifting theorem established the central position of stratified nilpotent Lie groups (nowadays known as Carnot groups) in the analysis of H\"ormander operators but, to present date, there exists no known counterpart of Jerison's results in these sub-Riemannian ambients. In this paper we fill this gap. We prove optimal $\Tau^{k,\alpha}$ ($k\geq 2$)  Schauder estimates near a $C^{k,\alpha}$ non-characteristic portion of the boundary for $\Tau^{k-2, \alpha}$ perturbations of horizontal Laplacians in Carnot groups. 
\end{abstract}

\maketitle

\tableofcontents

\section{Introduction and statement of the main result}
The fundamental role of Schauder estimates (both interior and at the boundary) in the theory of elliptic and parabolic partial differential equations is well-known. In this paper we are interested in $\Gamma^{k,\alpha}$ Schauder estimates at the boundary in the Dirichlet problem for a class of second order partial differential equations in stratified, nilpotent Lie groups, nowadays known as Carnot groups. The central position of such Lie groups in the analysis of the hypoelliptic operators introduced by H\"ormander in his famous paper \cite{H} was established in the 1976 work of Rothschild and Stein on the so-called \emph{lifting theorem}, see \cite{RS}. One should also see Stein's visionary address at the 1970 ICM in Nice \cite{Snice}, and also E. Cartan's seminal address at the 1928 ICM in Bologna \cite{Ca}. 

To provide the reader with some historical background on our main result we mention that in his 1981 works \cite{Je1, Je2} D. Jerison first analysed the question of Schauder estimates at the boundary for the horizontal Laplacian in the Heisenberg group $\Hn$ (see also \cite{Je3} for a further extension to CR manifolds)\footnote{Long known to physicists as the Weyl's group, $\Hn$ is an important model of a (non-Abelian) Carnot group of step $r=2$, see Definition \ref{D:CG} below for the general notion.}. Jerison divided his analysis into two parts, according to whether or not the relevant portion of the boundary contains so-called \emph{characteristic points}, a notion that goes back to the pioneering works of Fichera \cite{Fi1, Fi2} (see Definition \ref{D:char} below). At such points the vector fields that form the relevant differential operator become tangent to the boundary  and thus one should expect a sudden loss of differentiability,  somewhat akin to what happens in the classical setting with oblique derivative problems. Remarkably, Jerison proved that there exist no Schauder boundary estimates at characteristic points! He did so by constructing a domain in $\Hn$ with real-analytic boundary that support solutions of the horizontal Laplacian $\Delta_H u = 0$ vanishing near a characteristic boundary point, and which near such point possess no better regularity than H\"older. For a detailed analysis of these aspects in connection with the subelliptic Dirichlet problem we refer the reader to the papers \cite{LU, CG, GV, CGNajm, CGN}. 

On the positive side, Jerison proved in \cite{Je1} that at a non-characteristic portion of the boundary it is possible to develop a Schauder theory based on the non-isotropic Folland-Stein H\"older classes $\Gamma^{k,\alpha}$. He achieved this by adapting to $\Hn$ tools from Fourier and microlocal analysis. However, if we leave the Heisenberg group $\Hn$ and we move to a general Carnot group $\G$, then there exists no known counterpart of the results from \cite{Je1}, with the exception of the following partial ones. In \cite{BGM} we derived $\Gamma^{1, \alpha}$ boundary Schauder estimates near a $C^{1, \alpha}$ portion of the non-characteristic boundary and for variable coefficient operators such as \eqref{dp0} below. Our method employed a delicate adaptation of Caffarelli's compactness arguments.
In \cite{BCC},  via an adaptation of the classical method of the Levi parametrix, the authors obtained $\Gamma^{2, \alpha}$ boundary Schauder estimates  near a $C^\infty$ portion of the non-characteristic boundary, and for constant coefficient operators, under the additional geometric hypothesis that, roughly speaking, the boundary has a H\"ormander structure (see assumption (1.5) in \cite{BCC}). As it was noted in the subsequent related work \cite{CGS}, where the authors obtained an inversion result for a double layer potential at non-characteristic points in $\mathbb H^1$, even in the model Heisenberg group, their geometric hypothesis is not generically verified.

In the present work we develop a complete theory of the higher-order  boundary Schauder estimates for general Carnot groups.  Before we state our main result we introduce some notations, referring the reader to Sections \ref{S:prel} and \ref{S:back}
 for a detailed account.
For $k\in \mathbb N\cup\{0\}$ and $0<\alpha< 1$ we indicate with $\Gamma^{k,\alpha}$ the Folland-Stein non-isotropic H\"older classes, see Definition \ref{hf} and Definition \ref{higherholder} below. If $\Om\subset \G$ is a bounded open set, then given a point $p\in \pa \Om$, for any $s>0$ we set for simplicity
\begin{equation}\label{VS}
\V_s(p) = \Om\cap B(p,s),\ \ \ \ \ \ \mathscr S_s(p) = \pa\Om\cap B(p,s),
\end{equation}
where $B(p,s)$ is the open ball in the pseudodistance  \eqref{pseudo} below. When $p=e$, the group identity, we simply write $\V_s$ and $\mathscr S_s$. We suppose that $\mathbb{A}= [a_{ij}]$ be a given $m\times m$ symmetric matrix-valued function with real coefficients and satisfying the following ellipticity condition for some $\lambda > 0$, 
\begin{equation}\label{ea0}
 \lambda \mathbb{I}_m \leq \mathbb{A}(p) \leq \lambda^{-1} \mathbb{I}_m,\ \ \ \ \ \ \ \ \ p\in \G,
 \end{equation}
where $\mathbb{I}_m$ denotes the $m \times m$ identity matrix. The hypothesis \eqref{ea0} will be assumed throughout the paper, without further reference. Furthermore, for a given $k\in \mathbb N$ and a multi-index $I$ of order $|I|=k$,  the $k$-th order horizontal derivative $X^I f$  is defined  as in \eqref{hder} below. Our main result is as follows. 

\begin{theorem}\label{main}
Suppose that $u\in \Tau^2(\V_s(p_0)) \cap C(\overline{\V_s}(p_0))$ be a weak solution to 
\begin{equation}\label{dp0}
\sum_{i,j=1}^m  a_{ij}(p) X_iX_j u = f\ \ \ \text{in}\ \V_s(p_0),\ \ \ \ \ u  = \phi\ \ \ \text{on}\ \Scr_s(p_0),
\end{equation}
for given $p_0 \in \pa \Om$ and $s>0$. Assume that $\mathscr S_s(p_0)$ be non-characteristic and of class $C^{k,\alpha}$, for some $k \in \mathbb N$ such that $k\geq 2$ and $\alpha \in (0,1)$, and that 
\begin{equation}\label{assump1}
a_{ij} \in \Tau^{k-2, \alpha}(\overline{\V_s}(p_0)),\ \ f \in \Tau^{k-2,\alpha}(\overline{\V_s}(p_0)),\  \ \phi \in \Tau^{k,\alpha}(\overline{\V_s}(p_0)).
\end{equation} 
Then, $u\in \Tau^{k,\alpha}(\overline{\V_{s/2}}(p_0))$, and we have the following a priori estimate for the $k$-th order H\"older seminorm of $u$
\begin{align}\label{ap}
&\underset{|I| =k}{\sup} [X^I u]_{\Tau^{k,\alpha}(\V_{s/2}(p_0))} \\& \leq \frac{C}{s^{k+\alpha}} \bigg[||u||_{L^{\infty}(\V_s(p_0))} + \sum_{|I| \leq k-2} s^{2+|I|} ||X^I f||_{L^{\infty}(\V_s(p_0))} \notag \\ &+ \sum_{|I|=k-2} s^{k+\alpha} [X^If]_{\Tau^{0,\alpha}(\V_s(p_0))}+\sum_{|I| \leq k} s^{|I|} ||X^I \phi||_{L^{\infty}(\V_s(p_0))} + s^{k+\alpha} \sum_{|I|=k}  [X^I \phi]_{\Tau^{2,\alpha}(\V_s(p_0))} \bigg],\notag
\end{align}
where $C = C(\G, \lambda, \Om, k, \alpha, [a_{ij}]_{\Tau^{k-2,\alpha}(\V_s(p_0))})>0$.
\end{theorem}

\begin{remark}
Concerning Theorem \ref{main}, we emphasise that we have considered boundary value problems of the type \eqref{dp0} solely for the simplicity of exposition. With standard modifications, our method can be adapted  to infer Schauder estimates for more general equations of the type
\[
\sum_{i,j=1}^m a_{ij} X_i X_j u + \sum_{i=1}^m b_i X_i u + c u= f,
\]
where $b_i$'s and $c$ are in $\Gamma^{k-2, \alpha}(\overline{\V_s}(p_0))$.
\end{remark}

Our approach to Theorem \ref{main} is inspired to the groundbreaking 1989 paper of Caffarelli \cite{Ca}. In this connection, we mention that the homogeneous  structure of a Carnot group, based on the presence of a family of non-isotropic dilations, plays a critical role. We have also been influenced by De Silva and Savin's beautiful results on the higher-order boundary Harnack inequality and Schauder theory for slit domains, in connection with thin obstacle problems, see \cite{DS1, DS2}. Having said this, the reader will see that, in the present sub-Riemannian framework, our approach, which completely avoids flattening of the boundary, has entailed some very delicate adaptation of the Euclidean ideas in the above mentioned works. 

We next provide the reader with some perspective on Theorem \ref{main}. If $u$ is a solution to \eqref{dp0} vanishing on the boundary (after subtracting off the boundary datum), then via compactness arguments we deduce that, at a non-characteristic boundary point $p_0$, $u$ can be approximated up to order $"k+\alpha"$ by  $\Gamma^{k, \alpha}$ functions of the type $d P$, where $d$ is the signed Riemannian distance function from the boundary, and $P$ is an appropriate stratified polynomial of homogeneous degree at most $k-1$. This step crucially utilises, for the limiting ``constant coefficient problem", the non-characteristic boundary regularity of Kohn-Nirenberg in \cite{KN}. Finally, in our analysis the $k-$th order  Taylor polynomial of $dP$ at $p_0 \in \partial \Om$ takes the role of the $k-$th order Taylor polynomial of $u$ at $p_0$. We note here that, for the case $k=1$ treated in \cite{BGM}, our approach was instead to  approximate  the solution $u$ with affine functions that came from the limiting problem, and then at each step we had to  ensure  the $"1+\alpha"$  scale invariant decay of the boundary datum to eventually derive the existence of the first order limiting Taylor polynomial at the boundary point. 

Differently from \cite{BGM}, in the present work we   approximate the solutions  to \eqref{dp0} (corresponding to  $\phi=0$)  with $\Gamma^{k, \alpha}$ functions of the type $dP$   in order to ensure  homogeneous (zero) boundary conditions at each step, and then eventually we pass to the  $k$-th Taylor polynomial at the end of the iteration process. Although this gets rid of the difficulty of ensuring $k+\alpha$  decay of the boundary datum at every step, a serious new difficulty arises. In this approach the functions $dP$  are not in the kernel of the ``constant coefficient operator" $\sum_{i,j=1}^m a_{ij}(p_0) X_i X_j$. Therefore, in order to obtain the appropriate invariant decay of the right-hand side one has to incorporate the right correction factors. The latter turn out to be small perturbations of the Taylor polynomial of the limiting problem. This is precisely where we borrow some beautiful ideas from the works \cite{DS1, DS2}. In this endeavor the stratified nature of the Lie algebra plays a very crucial role and allows us to adapt the Euclidean type algorithms in \cite{DS1, DS2} to the sub-Riemannian framework. This represents one key novelty of our work. Once the $k+\alpha$ boundary decay is obtained, we combine it with the  interior Schauder estimates in \cite{CH, Xu} to deduce that the $k$-th order derivatives are H\"older continuous up to the boundary. In the present work, however, combining the interior estimates with the boundary decay is considerably subtler than the case $k=1$ treated in \cite{BGM}.

The paper is organised as follows. In Sections \ref{S:prel} and \ref{S:back} we introduce some basic notations and notions and gather some  known results that are relevant to our work. In Section \ref{S:main} we  prove our main result.

\section{Preliminaries on Carnot groups}\label{S:prel}

In this section we collect the definition and the basic properties of Carnot groups which will be used in the rest of the paper, referring the reader to \cite[Chap. 1 \& 2]{Gems} for a short and self-contained introduction to these geometric ambients. 

\begin{definition}\label{D:CG}
Given $r\in \mathbb N$, a \emph{Carnot group} of step $r$ is a simply-connected real Lie group $(\G, \circ)$ whose Lie algebra $\g$ is stratified and $r$-nilpotent. This means that there exist vector spaces $\g_1,...,\g_r$ such that:  
\begin{itemize}
\item[(i)] $\g=\g_1\oplus \dots\oplus\g_r$;
\item[(ii)] $[\g_1,\g_j] = \g_{j+1}$, $j=1,...,r-1,\ \ \ [\g_1,\g_r] = \{0\}$.
\end{itemize}
\end{definition}
The identity element in $\G$ will be routinely denoted by $e$. We assume that $\g$ is endowed with a
scalar product $\langle\cdot,\cdot\rangle$ with respect to which the vector spaces $\g_j's$, $j=1,...,r$, are mutually orthogonal.  We let $m_j =$ dim\ $\g_j$, $j=
1,...,r$, and denote by $N = m_1 + ... + m_r$ the topological
dimension of $\G$.  From the assumption (ii) on the Lie
algebra it is clear that any basis of the first layer $\g_1$ bracket generates the whole Lie algebra $\g$. Because of such special role $\g_1$ is usually called the horizontal
layer of the stratification. For ease of notation we henceforth write $m = m_1$. In the case in which $r =1$ we are in the Abelian situation in which $\g = \g_1$, and thus $\G$ is isomorphic to $\R^m$. We are primarily interested in the genuinely non-Riemannian setting $r>1$. The exponential map $\exp : \g \to \G$ defines an analytic
diffeomorphism of the Lie algebra $\g$ onto $\G$, see e.g. \cite[Sec. 2.10 forward]{V}, or also \cite{CGr}. Using such diffeomorphism, whenever convenient we will routinely identify a point $p = \exp \xi \in \G$ with its logarithmic image $\xi = \exp^{-1} p\in \g$. More explicitly, this means that if $p= \exp(\xi_1+...+\xi_r)$, where $\xi_1 = x_1 e_1+...+x_m e_m$, $\xi_2 = x_{2,1} e_{2,1}+...+x_{2,m_2}e_{2,m_2}$, ... , $\xi_r = x_{r,1} e_{r,1},...,x_{r,m_r} e_{r,m_r}$, we identify $p$ with 
\begin{equation}\label{p}
p \cong (\xi_1,...,\xi_r) = (x_1,...,x_m,x_{2,1},...,x_{r,1},...,x_{r,m_r}).
\end{equation}
At times, it will be expedient to break the variables by grouping those in the horizontal layer into a single one as follows
$x = x(p) \cong \xi_1 = (x_1,\ldots,x_m)$, and indicate with $y= y(p)$ the $(N - m)-$dimensional vector
\[
y \cong (\xi_2,\xi_3,\ldots,\xi_r) = (x_{2,1},\ldots,x_{2,m_2},\ldots,x_{r,1},\ldots,x_{r,m_r}).
\]
In this case, we will write $w = (x,y)$.
Given $\xi, \eta\in \g$, the Baker-Campbell-Hausdorff formula reads
\begin{equation}\label{BCH}
\exp(\xi) \circ \exp(\eta) = \exp{\bigg(\xi + \eta + \frac{1}{2}
[\xi,\eta] + \frac{1}{12} \big\{[\xi,[\xi,\eta]] -
[\eta,[\xi,\eta]]\big\} + ...\bigg)},
\end{equation}
where the dots indicate commutators of order four and higher, see \cite[Sec. 2.15]{V}. Since by (ii) in Definition \ref{D:CG} all commutators of order $r$ and higher are trivial, in every Carnot group the series in the right-hand side of \eqref{BCH} is finite. 
Using \eqref{BCH}, with $p = \exp \xi,  p' = \exp \xi'$, one can recover the group law $p \circ p'$ in $\G$ from the knowledge of the algebraic commutation relations between the elements of its Lie algebra. We respectively denote
by $L_p(p') = p \circ p'$ and $R_p(p') = p'\circ p$ the left- and right-translation operator by an element $p\in
\G$. We indicate by $dp$ the bi-invariant
Haar measure on $\G$ obtained by lifting via the exponential map
 the Lebesgue measure on $\g$. 

The stratification (ii) induces in $\g$ a natural one-parameter family of non-isotropic dilations by assigning to each
element of the layer $\g_j$ the formal degree $j$.
Accordingly, if $\xi = \xi_1 + ... + \xi_r \in \g$, with $\xi_j\in \g_j$,
one defines dilations on $\g$ by the rule
$\Delta_\lambda \xi = \lambda \xi_1 + ... + \lambda^r
\xi_r,$
and then use the exponential map to transfer such
anisotropic dilations to the group $\G$ as
follows
\begin{equation}\label{dil}
\delta_\lambda(p) = \exp \circ \Delta_\lambda \circ
\exp^{-1} p.
\end{equation}
Given $f:\G\to \R$, the action of $\{\delta_\lm\}_{\lm>0}$ on $f$ is defined by
\[
\delta_\lm f(p) = f(\delta_\lm(p)),\ \ \ \ \ \ \ \ \ \ \ p\in \G.
\]
A function $f:\G \to \R$ is called homogeneous of degree $\kappa\in \R$ if for every $\lm >0$ one has
$\delta_\lm f = \lm^\kappa f$.
A vector field $Y$ on $\G$ is homogeneous of degree $\kappa$ if for every $f\in C^\infty(\G)$ one has $Y(\delta_\lm f) = \lm^\kappa \delta_\lm(Yf)$.
The dilations \eqref{dil} are group automorphisms, and thus we have for any $p, p'\in \G$ and $\lambda>0$
\begin{equation}\label{ga}
(\delta_\lambda(p))^{-1} = \delta_\lambda(p^{-1}),\ \ \ \ \ \ \ \ \ \ \delta_\lambda(p) \circ \delta_\lambda(p') = \delta_\lambda(p\circ p').
\end{equation}

The homogeneous dimension of $\G$ with respect to \eqref{dil} is the number $Q = \sum_{j=1}^r j m_j.$ Such number plays an important role in the analysis of Carnot groups. 
The motivation for this name comes from the equation 
$$(d\circ\delta_\lambda)(p) = \lambda^Q dp,$$
where $dp$ denotes the push forward to $\G$ of Lebesgue measure on $\g$ (this defines a Haar measure on $\G$, see \cite{CGr}). In the
non-Abelian case $r>1$, one clearly has $Q>N$. A non-isotropic gauge in $\g$ which respects the dilations $\Delta_\lambda$, is given by $|\xi|_\g = \left(\sum_{j=1}^r ||\xi_j||^{2r!/j}\right)^{1/2r!}$, see \cite{F}. It is obvious that $|\Delta_\lambda \xi|_\g = \lambda |\xi|_\g$ for $\lambda>0$.
One defines a non-isotropic gauge in the group $\G$ by letting $|p| = |\xi|_\g$ for $p = \exp \xi$. Clearly, $|\cdot|\in C^\infty(\G\setminus\{e\})$, and moreover $|\delta_\lambda p| = \lambda |p|$ for every $p\in \G$ and $\lambda>0$. The pseudodistance 
\begin{equation}\label{pseudo}
d(p,p')  = |(p')^{-1} \circ p|,
\end{equation}
generated by such gauge is equivalent to the intrinsic, or Carnot-Carath\'eodory distance $d_C(g,g')$ on $\G$, i.e., there exists a universal constant $c_1>0$ such that for every $p, p'\in \G$
\begin{equation}\label{pseudoe}
c_1 d(p,p')  \le d_C(p,p') \le c_1^{-1} d(p,p'),
\end{equation}  
see \cite{F}, \cite{NSW}. We will almost always work with the pseudodistance \eqref{pseudo}, and denote by $B(p,s) = \{p'\in \G\mid d(p',p)<s\}$ the relative open balls. Hereafter, we indicate with $|E| = \int_E dg$ the Haar measure of a set $E\subset \G$.
One easily recognises that, with $\omega_C  = |B_C(1)|>0$ and $\omega = |B(1)|>0$, one has for every $p \in \G$ and $r > 0$, 
\begin{equation}\label{volb}
|B_C(p,r)| =\omega_C r^Q,\ \ \ \ \ \ \ \ \ \ \ |B(p,r)| =\omega r^Q.
\end{equation}
We will use the following result from \cite{NSW}. Denote by $d_e(p,p')$ the Riemannian distance in $\G$.   We will denote by $B_e(p,r)$, the corresponding ball  centered at $p$ of radius $r$  with respect to the Riemannian distance $d_e$. 
For every connected  $\Omega \subset\subset \G$ there exist $C,\varepsilon
>0$ such that for $p, p' \in \Omega$ one has
\begin{equation}\label{xy}
C d_e(p,p') \leq d_C(p,p')\leq C^{-1} d_e(p,p')^\varepsilon.
\end{equation}

Given a orthonormal basis $\{e_1,...,e_m\}$ of the horizontal layer $\g_1$ one associates corresponding left-invariant $C^\infty$ vector fields on $\G$ by the formula $X_i(p) = (L_p)_\star(e_i)$, $i=1,...,m$, where $(L_p)_\star$ indicates the differential of $L_p$. We note explicitly that, given a smooth function $u$ in $\G$, the derivative of $u$ in $p\in \G$ along the vector field $X_i$ is given by the Lie formula
\begin{equation}\label{lie}
X_i u(p)  = \frac{d}{dt} u(p \exp t e_i)\big|_{t=0}.
\end{equation}
The left-invariant vector fields defined by \eqref{lie} are homogeneous of degree $\kappa = 1$. We also note that, if $X_i^\star$ indicates the formal adjoint of $X_i$ in $L^2(\G)$, then $X_i^{\star} = -X_i$, see \cite{F}. If $\{e_{j,1},...,e_{j,m_j}\}$ denotes an orthonormal basis of the layer $\mathfrak g_j$, $j=2,...,r$, then as in \eqref{lie} we can define left-invariant vector fields $\{X_{j,k}\}_{k=1}^{m_j}$. With respect to this notation, we have the identification $X_{1,i}=X_i$, $i=1,...,m$. 

The \emph{carr\'e du champ} and the \emph{horizontal Laplacian} associated with the orthonormal basis $\{e_1,...,e_m\}$  are respectively given by
\begin{equation}\label{lhg}
|\nabla_{H} f|^2 = \sum_{j=1}^m (X_j f)^2,
\end{equation}
and
\begin{equation}\label{subl}
 \Delta_H = -\sum^m_{i=1}X^{\star}_{i}X_i=\sum^m_{i=1}X^2_i.
\end{equation}
It is clear that for every $f\in C^\infty(\G)$ one has
$\Delta_H(\delta_\lm f) = \lm^2 \delta_\lm(\Delta_H f)$.
Furthermore, by the assumptions (i) and (ii) in Definition \ref{D:CG} one immediately sees that the system $\{X_1,...,X_m\}$ satisfies the finite rank condition \begin{equation}\label{rank}\text{rank Lie} [X_1,\ldots,X_m] \equiv N, \end{equation}
therefore by  H\"ormander's theorem \cite{H} the operator $\Delta_H$ is hypoelliptic. However, when the step $r$ of $\G$ is $>1$ this operator fails to be elliptic at every point $p\in \G$.

In the proof of Theorem \ref{main} the specific structure of the $X_i$ in \eqref{lie} will be important, as well as the so-called stratified Taylor formula. With this in mind, we recall next some facts from \cite[Sec. C, p. 20]{FS}. A multi-index  $J\in (\mathbb N \cup \{0\})^N$ can be represented as 
\begin{equation}\label{I}
J = (\beta^J_1,\ldots,\beta^J_r) = (\beta^J_{1,1},\ldots,\beta^J_{1,m},\ldots,\beta^J_{r,1},\ldots,\beta^J_{r,m_r}).
\end{equation}  
Whenever the context is clear, for $i=1,...,r$, we will simply write $\beta_i$ instead of $\beta^J_i$. As customary, the length of $J$ is the number $|J| = \sum_{j=1}^r  |\beta_j| = \sum_{j=1}^r \sum_{s=1}^{m_j} \beta_{j,s}$. The \emph{weighted degree} of $J$ is instead defined by the equation
\begin{equation}\label{hl}
d(J) = \sum_{j=1}^r j |\beta_j| = \sum_{j=1}^r j (\sum_{s=1}^{m_j} \beta_{j,s}).
\end{equation}
Given $J$ as in \eqref{I}, and keeping \eqref{p} in mind, we consider the monomial $z^J$ defined by 
\[
z^J = \xi_1^{\beta_1} \cdots\ \xi_r^{\beta_r} = \prod_{j=1}^r x_{j,1}^{\beta_{j,1}} \cdots\ x_{j,m_j}^{\beta_{j,m_j}}.
\]
It is clear that the function $f(p) = z^J$ is homogeneous of weighted degree $\kappa = d(J)$. 
A \emph{stratified polynomial} in $\G$  is a function $P:\G\rightarrow \R$ which, in the logarithmic coordinates $z = (x,y)$, can be expressed as
\[
P(z)=\sum_J a_J z^J,
\]
where $a_J\in \R$. The \emph{weighted degree} of $P$ is the largest $d(J)$ for which the corresponding $a_J\not= 0$. For any  $\kappa\in \mathbb{N}\cup \{0\}$ we denote by $\mathscr{P}_\kappa$ the set of stratified polynomials in $\G$ of weighted degree less or equal to $\kappa$. 
The space $\mathscr{P}_\kappa$ is invariant under left- and right-translation. 

We recall that by the Baker-Campbell-Hausdorff formula \eqref{BCH} one obtains from \eqref{lie} the following expression 
\begin{equation}\label{fdG}
X_i  = \frac{\partial }{\partial{x_i}} +
\sum_{j=2}^{r}\sum_{s=1}^{m_j} b^s_{j,i}(\xi_1,...,\xi_{j-1})
\frac{\partial }{\partial{x_{j,s}}},
\end{equation}
where each $b^s_{j,i}$ is a homogeneous polynomial of
weighted degree $j-1$. This means that for $j=2,...,r$  we can write 
\[
b^s_{j,i}(\xi_1,...,\xi_{j-1})=\sum_{ d(I^{(j-1)})= j-1}b^s_{i,I^{(j-1)}}z^{I^{(j-1)}},
\]
where $I^{(j-1)}$  denotes a multi-index with zeroes in the last $N- \sum_{i=1}^{j-1} m_{i}$ coordinates, and $b^s_{i,I^{(j-1)}}\in \R$. More explicitly, if $I^{(j-1)}=(\beta_1,\ldots,\beta_{(j-1)},0,\ldots,0)$, then $z^{I^{(j-1)}}=\xi_1^{\beta_1}\cdots\xi^{\beta_{j-1}}_{j-1}$, and \eqref{fdG} can be alternatively written as 
\begin{equation}\label{xiex2}
X_i= \frac{\partial}{\partial x_i}+\sum^r_{j=2}\sum^{m_j}_{s=1}\left(\sum_{ d(I^{(j-1)})= j-1}b^s_{i,I^{(j-1)}}z^{I^{(j-1)}}\right)\frac{\partial}{\partial x_{j,s}}.
\end{equation}
Using \eqref{xiex2} it will be useful in the proof of Theorem \ref{main} to have the following expanded expression of \eqref{subl}. For $j\in\{1,...,r\}$ and $s\in \{1,...,m_j\}$, we will denote by $\overline{(j,s)}$ a multi-index which has 1 in the $(j,s)$ position and zeroes everywhere else.  When $j=1$, for convenience we will denote $\overline{(1,i)}$ with $\overline{i}$. With this notation in place, a tedious but trivial bookkeeping based on \eqref{xiex2} gives for \eqref{subl}
\begin{align}\label{subLap}
\Delta_H
&= \sum^{m}_{i=1}\frac{\partial^2}{\partial x^2_i}+2\sum^{m}_{i=1}\sum^r_{j=2}\sum^{m_j}_{s=1}\left(\sum_{d(I^{(j-1)})= j-1}b^s_{i,I^{(j-1)}}z^{I^{(j-1)}}\frac{\partial^2}{\partial x_{j,s}\partial x_i}\right)
\\
 & + \sum^{m}_{i=1}\sum^r_{j=2}\sum^{m_j}_{s=1}\left(\sum_{d(I^{(j-1)})= j-1}\alpha_{1,i} b^s_{i,I^{(j-1)}}z^{I^{(j-1)- \overline{i}}}\frac{\partial}{\partial x_{j,s}}\right)
 \notag\\ 
 &+\sum^{m}_{i=1}\sum^r_{j,j'=2}\sum^{m_j}_{s,s'=1}\left(\sum_{\substack{d(I^{(j-1)})= j-1\\  d(I^{(j'-1)})= j'-1 }}b^s_{i,I^{(j-1)}}b^{s'}_{i,I^{(j'-1)}}z^{I^{(j-1)}+I^{(j'-1)}}\right)\frac{\partial^2}{\partial x_{j,s}\partial x_{j',s'}}
\notag \\
&+\sum^{m}_{i=1}\sum^r_{\substack{j,j'=2\\ j'\leq j-1}}\sum^{m_j}_{s,s'=1}\left(\sum_{\substack{d(I^{(j-1)})= j-1\\  d(I^{(j'-1)})= j'-1 }}\alpha_{j',s'}b^s_{i,I^{(j-1)}}b^{s'}_{i,I^{(j'-1)}}z^{I^{(j-1)}-\overline{(j',s')}+I^{(j'-1)}}\right)\frac{\partial}{\partial x_{j,s}}.
\notag
\end{align}
We mention that the reason for introducing the expanded expression \eqref{subLap} of $\Delta_H$ is because, in \emph{Step 1} of the proof of Theorem \ref{main}, we eventually have to determine an appropriate approximating polynomial  by solving a linear system of the type \eqref{alg} below. In order to do so, it is crucial to have a precise knowledge regarding  the value of the $(1,m)$ coordinate of the various multi-indices  involved in the  expression \eqref{alg}.

%%%%%%%%%%%%%%%%%%%%%%%%%%%%%%%%%%%%%%

\section{Further background material}\label{S:back}
 
In this section we collect some further known results that will be needed in the proof of Theorem \ref{main}. We begin with the following weak maximum principle, see \cite{BU}.

\begin{lemma}\label{comp}
Let $\Om\subset \G$ be a bounded open set, and $a_{ij}, b_i\in C(\Om)$. If $u\in \Gamma^2(\Om)\cap C(\overline \Om)$ satisfies $\sum_{i,j=1}^m a_{ij}(p) X_i X_j u + \sum_{i=1}^m b_i(p) X_i u \ge 0$ in $\Om$ and $u\le 0$ on $\partial \Om$, then $u\le 0$ in $\Om$.
\end{lemma}

Another basic tool in this paper are the intrinsic H\"older classes $\Tau^{k, \alpha}$ of those functions that are H\"older continuous, together with their derivatives up to order $k$ along the vector fields \eqref{lie}, with respect to the Carath\'eodory distance $d_C(p,p')$, or equivalently with respect to the pseudodistance $d(p,p')$ in \eqref{pseudo}, see \eqref{pseudoe}. For the proofs of the relevant results we refer the reader to \cite[Sec. C on p. 20]{FS} and also \cite[Chap. 20]{BLU}.

\begin{definition}\label{hf}
Let $0<\alpha\leq 1$. Given an open set $\Om \subset \G$ we say that $f:\Om\to \R$ belongs to $\Tau^{0, \alpha}(\Om)$ if $f\in L^\infty(\Om)$ and there exists a constant $M>0$ such that for every $p, p' \in \Om$ one has 
$|f(p) - f(p')| \leq M\ d(p,p')^{\alpha}$.
We define the seminorm 
\begin{equation}\label{semi}
[f]_{\Tau^{0,\alpha}(\Om)}= \underset{\underset{p \neq p'}{p, p'\in \Om}}{\sup} \frac{|f(p)-f(p')|}{d(p,p')^{\alpha}}.
\end{equation}
\end{definition}
From \eqref{pseudoe} and \eqref{xy} it is obvious that if $f\in \Tau^{0, \alpha}(\Om)$, then $f\in C(\Om)$.
We now define the higher order H\"older spaces. If $k\in \mathbb N$ and $I = (i_1,...,i_k)$, where $1\le i_j\le m$, and $j=1,...,k$. We set $|I| = k$, and let 
\begin{equation}\label{hder}
X^I = X_{i_1}...X_{i_k}.
\end{equation}

\begin{definition}\label{higherholder}
Let $k \in \mathbb N \cup \{0\}$, $0 < \alpha \leq 1$ and $\Om \subset G$ be an open connected set. The space $\Gamma^{k, \alpha}(\Om)$ denotes the space of all functions $f$ for which $X^{I} f$ exists for $|I| \leq k$ and such that the following quantity is finite
\[
||f||_{\Gamma^{k, \alpha}(\Om)} \overset{def}= \sum_{|I| \leq k} ||X^If||_{L^{\infty}(\Om)} + \sum_{|I|=k} [X^I f]_{\Gamma^{0, \alpha}(\Om)} < \infty.\]
\end{definition}

\subsection{Taylor polynomials and stratified Taylor inequality}
We first introduce the notion of a Taylor polynomial, see \cite[pages 26-27]{FS}.
\begin{definition}\label{tp}
Let $f$ be a function whose derivatives $X^{I} f$ are continuous for $|I| \leq k$. The $k$-th Taylor polynomial of $f$ at $g_0 \in \G$ is the unique polymomial $P_{g_0}$ of homogeneous degree at most $k$ such that
\[
X^I f(g_0) =  X^I P_{g_0} (e),
\]
for all $I$ such that $|I| \leq k$.

\end{definition}
\begin{remark}
We mention that in view of the bracket generating condition in \eqref{rank} above,  the notion of a Taylor polynomial in Definition \ref{tp} is the same  as  that in \cite[pages 26-27]{FS}. \end{remark}

We now state the relevant stratified Taylor inequality which can be found in \cite[pages 33-35]{FS}.
\begin{theorem}[Stratified Taylor inequality]\label{stp}
For each positive integer $k$, there exist constants $c_k, b>0$ such that for all $f$ for which $X^{I} f$ is continuous for $|I| \leq k$, one has
\begin{equation}\label{stp1}
|f(g_0 g) - P_{g_0} (g)| \leq c_k |g|^k \operatorname{sup}_{|p| \leq b^k |g|, |I|=k} |X^I f(g_0 p ) - X^I f(g_0)|,
\end{equation}
where $P_{g_0}$ is the $k$-th order Taylor polynomial of $f$ at $g_0$. 
\end{theorem}

\subsection{The characteristic set}\label{SS:char}

We recall that an open set $\Om \subset \G$ is said to be of class $C^1$ if for every $p_0\in \pa\Om$ there exist a neighborhood $U_{p_0}$ of $p_0$, and a function $\vf_{p_0} \in C^1(U_{p_0})$, with $|\nabla \vf_{p_0}| \geq \alpha > 0$ in $U_{p_0}$, such that
\begin{equation}\label{ome}
\Om \cap U_{p_0} = \{p\in U_{p_0} \mid \vf_{p_0}(p)<0\}, \quad\pa\Om\cap U_{p_0} =\{p\in U_{p_0} \mid \vf_{p_0}(p)=0\}.
\end{equation}
At every point $p\in \partial \Om \cap U_{p_0}$ the outer unit normal is given by
\[
\nu(p) = \frac{\nabla \vf_{p_0}(p)}{|\nabla \vf_{p_0}(p)|},
\] 
where $\nabla$ denotes the Riemannian gradient.

\begin{definition}\label{D:char}
Let  $\Om \subset \G$ be an open set of class $C^1$. A point $p_0\in \pa\Om$ is called \emph{characteristic} if  one has
\begin{equation}\label{cp}
\nu(p_0) \perp \operatorname{linear\ span\ of}\ \{X_1,..,X_m\}. 
\end{equation}
The \emph{characteristic set} $\Sigma = \Sigma_{\Om}$ is the collection of all characteristic points of $\Om$. A boundary point $p_0 \in \pa \Om \setminus \Sigma$ will be referred to as a \emph{non-characteristic} boundary point. 
\end{definition}

We note explicitly that \eqref{cp} is equivalent to saying that,
given any couple $U_{p_0}, \vf_{p_0}$ as in \eqref{ome}, and the family of horizontal vector fields $\{X_1,...,X_m\}$ as in \eqref{lie} above, one has 
\begin{equation}\label{cp2}
X_1\vf_{p_0}(p_0) = 0, \ldots, X_m\vf_{p_0}(p_0) = 0,
\end{equation}
or equivalently
\begin{equation}\label{cp3}
 |\nabla_H \vf_{p_0}(p_0)| = 0.
\end{equation}
Bounded domains generically have non-empty characteristic set if they have a trivial topology. For instance, in the Heisenberg group $\Hn$, with logarithmic coordinates $(x,y,t)\in \R^{2n+1}$, every bounded $C^1$ open set which is homeomorphic to the sphere $\mathbb S^{2n}\subset \R^{2n+1}$ must have at least one characteristic point. A torus obtained by revolving around the $t$-axis a closed simple curve in $\mathbb H^1$ non intersecting the $t$-axis itself, provides an example of a non-characteristic domain. For instance, the domain bounded by $\{(x,y,t)\in \mathbb H^1\mid 16 t^2 + (x^2 +y^2 - 2)^2 = 1\}$ is non-characteristic. For a discussion of these aspects we refer the reader to \cite{CG} and \cite{CGN}. Examples of unbounded non-characteristic domains are the following. In $\mathbb H^1$ consider the unbounded domain $\Om= \{(x,y,t)\in \mathbb H^1\mid x < yt\}$. Its boundary $S = \partial \Om$ is an entire graph which is non-characteristic. We can in fact describe $S$ using the global defining function $\vf(x,y,t) = x - yt$, for which we have $X_1 \phi = \phi_x - \frac{y}2 \phi_t = 1 + \frac{y^2}2\ge 1$, for every $(x,y,t)\in S$. In view of \eqref{cp2} this shows that $\Sigma_\Om = \varnothing$. Furthermore, $S$ is an $H$-minimal surface that was shown to be unstable in \cite{DGN}. If instead, for a fixed vector $a \in \g_1 \setminus \{0\}$ and for $\lambda \in \R$, we consider the so-called \emph{vertical half-space}
 \[
H^+_a = \{p\in\G\mid \langle x(p),a\rangle > \lambda\}.
 \]
Then, using the global defining function $\vf(p) = \lm - \langle x(p),a\rangle$, one has $|\nabla_H \vf(p)| \equiv |a|>0$, and thus in view of \eqref{cp3} one has $\Sigma = \Sigma_{H^a_+} = \varnothing$. In $\mathbb H^n$ such half-spaces were shown to be stable $H$-minimal surfaces in \cite{DGNP}.

In the proof of Theorem \ref{main} we will need the following by now classical smoothness result at non-characteristic points of Kohn and Nirenberg, see \cite{KN} and also \cite{De}. We recall the notation \eqref{VS}.

\begin{theorem}\label{KN}
  Assume that $\Om$ be a $C^\infty$ domain, and let $u\in \mathscr{L}^{1,2}_{loc}(\Om)\cap C(\overline{\Om})$ be a weak solution of 
$\Delta_{H} u=0$ in $\mathscr V_{r_0}(p_0)$, $|u| \leq M$, $u=0$ on $\mathscr S_{r_0}(p_0)$, where $\mathscr S_{r_0}(p_0)$ consists of non-characteristic  boundary points. Then given any $k \in \mathbb N$,  there exists  a positive constant $C^\star=C^\star(\Om, M, p_0, k)>0$ such that
 \begin{equation}
 \label{c2u}
 \|u\|_{C^k(\overline{\Om}\cap V)}\leq C^\star.
 \end{equation}
\end{theorem}

We will also need the following interior regularity result which can be found in \cite{CH, Xu}.
\begin{theorem}\label{intreg}
Let $u$ be a solution to $\sum_{i,j=1}^m a_{ij}(p) X_i X_j u=f$
in a gauge ball $B(p_0,1)$. Suppose furthermore that for some $k \in \mathbb N$ such that $k \geq 2$ and  $0<\alpha<1$,  one has $a_{ij}, f \in \Gamma^{k-2, \alpha}(B(p_0,1))$. Then $u \in \Gamma^{k,\alpha}(B(p_0,1/2))$, and there exists a constant $C = C(\G,\lambda,[a_{ij}]_{\Gamma^{k-2, \alpha}(B(p_0,1))})>0$ such that the following estimates hold
\begin{equation}\label{i10}
 \underset{|J|=k}{\sup} ||X^J u||_{L^{\infty}(B(p_0,1/2))}  \leq  C \left\{||u||_{L^{\infty}(B(p_0,1))} + \sum_{|I| \leq k-2} ||X^I f||_{L^{\infty}(B(p_0, 1))} + \sum_{|I|=k-2} [X^I f]_{\Gamma^{0, \alpha}(B(p_0, 1))} \right\},
\end{equation}
and
\begin{equation}\label{i100}
\underset{|J|= k}{\sup} [X^J u]_{\Gamma^{0,\alpha}(B(p_0,1/2))}  \leq C \left\{||u||_{L^{\infty}(B(p_0,1))} + \sum_{|I| \leq k-2} ||X^I f||_{L^{\infty}(B(p_0, 1))} + \sum_{|I|=k-2} [X^I f]_{\Gamma^{0, \alpha}(B(p_0, 1))} \right\}, 
\end{equation}
where  $[\ \cdot\ ]_{\Gamma^{0, \alpha}(B(p_0,1))}$ represents the  seminorm  defined by \eqref{semi} above.
\end{theorem}
As a corollary, we have the following rescaled estimate.

\begin{corollary}\label{intreg1}
Let $u$ be a solution to 
$\sum_{ij} a_{ij}(p) X_i X_j u=f$
in a gauge ball $B(p_0,s)$. Suppose furthermore that for some $k \in \mathbb N$ such that $k \geq 2$ and  $0<\alpha<1$,  one has $a_{ij}, f \in \Gamma^{k-2, \alpha}(B(p_0,s))$. Then $u \in \Gamma^{k,\alpha}(B(p_0,s/2))$, and there exists a constant $C = C(\G,\lambda,[a_{ij}]_{\Gamma^{k-2, \alpha}(B(p_0,1))})>0$ such that the following estimates hold
\begin{align}\label{rei1}
& \operatorname{sup}_{|J|=k} ||X^J u||_{L^{\infty}(B(p_0,s/2))} \\& \leq  \frac{C}{s^k} \left\{||u||_{L^{\infty}(B(p_0,s))} + \sum_{|I| \leq k-2} s^{2+|I|} ||X^I f||_{L^{\infty}(B(p_0, s))} + \sum_{|I|=k-2} s^{2+|I|+\alpha}[X^I f]_{\Gamma^{0, \alpha}(B(p_0, s))} \right\}\notag,
\end{align}
and
\begin{align}\label{rei2}
&\underset{|J|= k}{\sup} [X^J u]_{\Gamma^{0,\alpha}(B(p_0,s/2))}\\&  \leq \frac{C}{s^{k+\alpha}} \left\{||u||_{L^{\infty}(B(p_0,s))} + \sum_{|I| \leq k-2} s^{2+|I|} ||X^I f||_{L^{\infty}(B(p_0, s))} + \sum_{|I|=k-2} s^{2+|I|+\alpha}[X^I f]_{\Gamma^{0, \alpha}(B(p_0, s))} \right\}.\notag
\end{align}
\end{corollary}
\begin{proof}
The proof of \eqref{rei1} and \eqref{rei2} is analogous to that of \cite[Corollary 3.2]{BGM} and therefore we skip the details.
\end{proof}
%%%%%%%%%%%%%%%%%%%%%%%%%%%%%%%%%%%%%%%%%%%%%%%%%%%%%%%
\section{Proof of Theorem \ref{main}}\label{S:main}

We begin with the following Lipschitz regularity result near a portion of the non-characteristic boundary. This constitutes the equicontinuous estimate that is required  in our compactness argument in the proof of Lemma \ref{MCL} below. Henceforth, given a symmetric uniformly elliptic matrix $\mathbb{A}=[a_{ij}]$ satisfying \eqref{ea0} for some $\lambda >0$ and all $p\in\Omega$,
 we use the notation
\begin{equation}\label{L}
\mathscr L\overset{def}{=}  \sum_{i,j=1}^m a_{ij}(p) X_iX_j. 
\end{equation}

\begin{proposition}\label{Lip}
Let $\Om\subset \mathbb{G}$ be a $C^{k,\alpha}$ domain, $k\geq 2$,  such that $p_0  \in \pa \Om$ is a non-characteristic point.   Let $u\in \Gamma^{2}(\Om)\cap C(\overline{\Om})$ be a solution of $\mathscr L u = f$ in $\Omega$, with $u=0$ on $\partial \Omega$.
Furthermore, assume that $a_{ij}$ and $f$ are in $ \Gamma^{\alpha}(\overline{\Om})$ for some $\alpha>0$. Then there exist $r_0$ and $C>0$,  depending on $\Om, \lambda, \alpha$, such that  
 \begin{equation}\label{lipunif}
\underset{\underset{p\neq p'}{p ,p'\in\overline{\Om \cap B_e(p_0,r_0)}}}{\sup} \frac{|u(p)-u(p')|}{d(p,p')} < C ( \|f\|_{L^{\infty}} + \underset{\Om}{\sup}\ |u|).\end{equation}
\end{proposition}
\begin{proof}
We let  $M \overset{def}{=}\underset{\Om}{\sup}\ |u|.$ Since $p_0$ is non-characteristic, by using the $C^{2}$ character of $\Om$ we can find a sufficiently small $0<r_1<1$, a local defining function $\vf\in C^2(B_e(p_0,r_1))$ for $\Om$, and a number $\gamma_0>0$ such that
\begin{equation}\label{unf0}
|\nabla_H \vf(p)| \ge \gamma_0
\end{equation}
for every $p\in \partial \Om \cap B_e(p_0,r_1)$. Notice that, in view of \eqref{cp3}, the condition \eqref{unf0} implies that all points $p\in \pa \Om \cap B_e(p_0, r_1) $ are non-characteristic. We can further assume that for every such $p$ there exists a point $\tilde p\in \G\setminus \overline \Om$ such that the Riemannian ball $B_e(\tilde p,r_1)$ be tangent from outside to $\pa \Om$ in $p$. This means that 
\begin{equation}\label{unf1}
\overline{B}_e(\tilde p,r_1) \cap \overline{\Om}= \{p\}.
\end{equation}
Moreover  by choosing a  smaller $\gamma_0$ if needed, we can ensure that  for all $q \in \Om\cap \pa B_e(p,r_1)$ one has
\begin{equation}\label{unf2}
d_e(q,\tilde p) \geq (1+\gamma_0) r_1.
\end{equation}
Given any $p\in \partial \Om \cap B_e(p_0,r_1)$, consider the function $\psi(q) = d_e(q,\tilde p)^2$, where $\tilde p$ is as in \eqref{unf1}. Notice that it must be $\nabla \psi(p) = \alpha(p) \nu(p)$, for some $\alpha(p)\in \R$. Since $\nu(p)$ is also parallel to $\nabla \vf(p)$, we conclude that $\nabla \psi(p) = \beta(p) \nabla \vf(p)$ for some $\beta(p)\in \R$ such that $|\beta(p)| \ge \ve_0>0$. By this observation and \eqref{unf0}, by possibly choosing a smaller $r_1$, we infer the existence of $\delta_0>0$ such that for all $p \in \Om \cap B_e(p_0,r_1)$ one has
\begin{equation}\label{t1}
|\nabla_H \psi(p)| = |\beta(p)| |\nabla_H \phi(p)| \geq \delta_0.
\end{equation}
In order to establish \eqref{lipunif} we only need to show that, given any $p \in \pa \Om \cap B_e(p_0,r_1)$, the following estimate hold for all $q \in \Om\cap  B_e(p,r_1)$ 
\begin{equation}\label{lipu}
|u(q)|\leq K_1 d_e(q,p),
\end{equation}
where $K_1$ depends on $r_1, \gamma_0, \delta_0$  as in \eqref{unf0}, \eqref{t1}. By continuity, it suffices to show \eqref{lipu} for  $p= p_0$.  We denote by $\tilde p_0\in \G\setminus \overline \Om$ the point corresponding to $p_0$ as in \eqref{unf1} above, and we now let $\psi(q) =  d_e(q,\tilde p_0)^2$. We note that $\psi\in C^\infty(\G)$. We next  introduce the function
\begin{equation}\label{funf}
g(q) \overset{def}{=} \left(\overline{M} + ||f||_{\infty} \right) \left(1-\left(\frac{r_1^2}{\psi(q)}\right)^k\right)
\end{equation}
for $q\in \Om\cap B_e(p_0,r_1),$ where $\overline{M},k$ and $L$  will be chosen later. We clearly have 
\begin{equation}\label{fgu1}
g\ge 0 = u\quad \ \ \ \ \ \ \text{on} \  \pa\Om\cap B_e(p_0,r_1).
\end{equation}
If instead $q\in \Om\cap \pa B_e(p_0,r_1)$, then it follows from \eqref{unf2}  that
\[
g(q) \geq \overline{M}\left(1-\left(\frac{1}{(1+\gamma_0)^2}\right)^k\right) \geq \frac{\overline{M}}{2},
\]
provided that $k\ge k_0 \overset{def}{=} \frac{\log 2}{\log(1+\gamma_0)}$. If we thus let $\overline{M}= 2 M$ and $k\ge k_0$, then we conclude that
\begin{equation}\label{fgu2}
g\ge u \quad \ \ \ \ \ \ \ \ \ \ \ \ \text{on}\  \Om\cap \pa B_e(p_0,r_1).
\end{equation}
If we could show that $\Ls g\leq \Ls u$ in $\Om\cap  B_e(p_0,r_1)$, then from (\ref{fgu1}), (\ref{fgu2}) and the comparison principle in Lemma \ref{comp}, we can conclude that 
\begin{equation}\label{bony}
u\leq g \quad \ \ \ \ \ \ \ \ \text{in}\ \Om\cap  B_e(p_0,r_1).
\end{equation}
We now observe that if $h\in C^2(\R)$, $v\in C^2(\G)$, then
\[
\Ls h(v) = h'(v)\Ls v +h''(v)\sum^m_{i,j=1}a_{ij}X_ivX_jv.
\]
Using this formula with $v=\psi$ and $h(t)= \overline M\left(1-(r_1^2/t)^k\right)$, we obtain
\begin{eqnarray*}
\Ls g &=&\left( \overline M + ||f||_{\infty}\right) k\frac{r_1^{2k}}{\psi^{k+2}}\left(\psi \Ls \psi - (k+1)\sum^m_{i,j=1}a_{ij}X_i\psi X_j\psi\right)\\
&\leq &\left( \overline M + ||f||_{\infty}\right) k\frac{r_1^{2k}}{\psi^{k+2}}\left(\psi \Ls \psi - \lambda(k+1)|\nabla_{H}\psi|^2\right).
\end{eqnarray*}
Note that in $\Om\cap  B_e(p_0,r_1)$ we have from \eqref{t1}  
\[
|\nabla_H \psi|^2 \ge \delta_0^2.
\]
Therefore, if $k$ is large enough, one can ensure that in $\Om\cap  B_e(p_0,r_1)$
\[
\Ls g \leq  - ||f||_{\infty}.
\]
We thus have  $\Ls g \leq \Ls u$ in $\Om\cap  B_e(p_0,r_1)$, and thus \eqref{bony} holds. In a similar way, if we take $-g$ as a lower barrier we can ensure that
\begin{equation}\label{bony2}
-g \leq u \quad\ \ \ \ \ \ \ \ \  \text{in} \ \Om\cap  B_e(p_0,r_1).
\end{equation}
Combining \eqref{bony} and \eqref{bony2}, we conclude
\[
|u|\leq g \quad \ \ \ \ \ \ \ \ \  \text{in} \ \Om\cap  B_e(p_0,r_1).
\]
Since $\psi\ge r_1^2$ in $\Om\cap  B_e(p_0,r_1)$, by \eqref{funf} we see that the function $g$ is Lipschitz continuous in $\Om\cap  B_e(p_0,r_1)$. From this fact, and $u(p_0)=g(p_0)=0$, we thus conclude that there exists a constant $K_1>0$ such that \eqref{lipu} holds for $p=p_0$, and therefore for all $p\in \pa \Om\cap  B_e(p_0,r_1)$, with $q \in \Om\cap  B_e(p,r_1)$.

If we now combine \eqref{lipu} with the left-hand side of \eqref{xy}, we can assert that there exists a constant $K_2>0$ such that 
\begin{equation}\label{lipu2}
|u(q)|\leq K_2 d_C(q,p) \quad \text{for all}\ q \in \Om\cap  B_e(p,r_1),\ p\in \pa \Om\cap  B_e(p_0,r_1)
\end{equation}
where $K_2>0$ depends on $r_1, \gamma_0, \delta_0$  as in \eqref{unf0}, \eqref{t1}. Now the desired estimate \eqref{lipunif} follows by combining the boundary decay estimate \eqref{lipu2} with the interior estimates in \cite{CH}.

\end{proof}

%%%%%%%%%%%%%%%%%%%%%%%%%%%%%%%%%%%%%%%%%%

\subsection{Proof of Theorem \ref{main}}
Recall that we are assuming that $\Om \in C^{k, \alpha}$, and that, for $s>0$ we have agreed to write $\V_s = \V_s(e)$ and $\mathscr S_s = \mathscr S_s(e)$. After subtracting off the boundary datum $\phi$ and a left-translation, it suffices to  look at the following boundary value problem
\begin{equation}\label{bp}
\mathscr{L} u = f\ \text{in}\ \V_s,\ \ \ \ \ \ \ \ u=0\ \text{on}\ \mathscr S_s,
\end{equation}
where $\mathscr Lu$ is as in \eqref{L},
with $a_{ij}, f \in \Tau^{k-2, \alpha}(\overline{\Om})$.    Then, by scaling with respect to the family of dilations $\{\delta_{\lambda} \}_{\lambda>0}$ in \eqref{dil}, and an appropriate  change of coordinates in the horizontal layer $\g_1$, we may also assume without loss of generality that $s=1$ and:
\begin{equation}\label{smn} 
a_{ij}(e)=\delta_{ij},\ \ \ \ \  ||a_{ij} - \delta_{ij}||_{\Tau^{k-2, \alpha}}, ||f||_{\Tau^{k-2, \alpha}} \leq \delta.\end{equation}
Furthermore  we can also assume that
 $\V_1$ can be expressed in the logarithmic coordinates as 
\begin{equation}\label{omega}
\Big\{(x',x_m,y)\mid x_m>h(x',y)\Big\}, \quad \text{with $h(0,0)=0$,  $\nabla_{x'}h(0,0)=0$},
\end{equation}
where $y=(\xi_2, ..., \xi_r)$, and $h\in C^{k,\alpha}$.
In what follows we will work with the distance $d_e(p,q)$ in $\G$. By this we mean the push forward  of the Euclidean distance on Lie algebra via the exponential map.
To simplify the notation we will write $\mathrm{dist}\,(p,\partial\Omega) = \inf \{d_e(p,q)\mid q\in\partial\Omega\}$, 
instead of $\operatorname{dist}_e(p,\Om)$, and accordingly write $d(p) = \mathrm{dist}\,(p,\partial\Omega), p\in B(1)$, instead of $d_e(p)$. 
In logarithmic coordinates we have
\begin{equation}\label{distres}
d(p) = |\nabla_{H} d(e)| x_m + P_d + d_0(p),
\end{equation}
where $P_d$ is a polynomial of homogeneous degree at least 2 and $d_0 \in C^{k, \alpha}$ with $|d_0(p)|= O(|p|^{k+\alpha})$. Formula \eqref{distres} will play an important role and requires an explanation. The inner unit normal along the graph is given by 
\[
\nu(x',x_m,y) = \left(- \frac{\nabla_{x'} h}{\sqrt{1+|\nabla_{(x',y)} h|^2}}, \frac{1}{\sqrt{1+|\nabla_{(x',y)} h|^2}}, - \frac{\nabla_{y} h}{\sqrt{1+|\nabla_{(x',y)} h|^2}}\right),
\]
where we have let $h = h(x',x_m,y)$. Since by \eqref{omega} we have $\nabla_{x'} h(0,0) =0\in \R^{m-1}$, the previous formula, combined with the fact $\nabla d(0) = \nu(0)$, gives
\[
\nabla d(0) = \left(0, \frac{1}{\sqrt{1+|\nabla_{y} h(0)|^2}}, - \frac{\nabla_{y} h(0)}{\sqrt{1+|\nabla_{y} h(0)|^2}}\right).
\]
In view of \eqref{fdG} this gives
\[
X_1 d(0) = 0,..., X_{m-1} d(0) = 0,\ X_m d(0) = \frac{1}{\sqrt{1+|\nabla_{y} h(0)|^2}},
\]
and therefore
\[
\nabla_H d(0) = \left(0, \frac{1}{\sqrt{1+|\nabla_{y} h(0)|^2}}\right).
\]
Finally, since $d\in C^{k,\alpha}$ (see \cite[Lemma 14.16]{GT}), from Taylor formula we have
\[
d(p) = d(0) + \langle\nabla d(0),p\rangle + O(|p|^2)
= x_m |\nabla_H d(0)| + \langle\nabla_y d(0),y\rangle + \text{higher order terms},
\]
from which \eqref{distres} follows. 
Having said this, for every $0<\sigma\leq 1$ consider  the domain  $\Omega_{\sigma}=  \delta_{\sigma^{-1}}(\Omega)$. We set
\[
\U_\s \overset{def}{=}  \Om_\s \cap B(\s^{-1}),\ \ \ \ \mathscr T_\s\overset{def}{=}  \pa \Om_\s \cap B(\s^{-1}).
\]
In logarithmic coordinates $\Omega_{\sigma}$ is given by
\begin{equation}\label{an1}
\Omega_{\sigma}=\{(x',x_m,y_2,\ldots,y_r)\mid (\sigma x',\sigma x_m,\sigma^2y_2,\ldots,\sigma^ry_r)\in\Omega\}.
\end{equation}
Observe that $\partial\Om_\s$ is given by 
\begin{equation}\label{xms}
x_m=h_{\s}(x',y)=h_{\s}(x',y_2,\ldots,y_r)\overset{def}{=}\frac 1\s h(\s x',\s^{2}y_2,\ldots,\s^{r} y_r).
\end{equation}
With $\Om_\sigma$ as in \eqref{an1} we let $d_{\sigma}(p) = \inf \{d_{e}(p,q)\mid q\in\partial\Omega_\sigma\}$.
From \eqref{omega} and \eqref{xms} it follows that given $\delta>0$, if $\sigma$ is small enough, one can ensure that 
\begin{equation}\label{sm0}
|| h_{\sigma}||_{C^{k, \alpha}(B(r))} \leq \delta,
\end{equation}
which, combined with \eqref{distres}, implies that 
\begin{equation}\label{sm1}
||d_{\sigma}- |\nabla_{H} d_{\sigma}(e)| x_m||_{\Gamma^{k, \alpha}(B(1))} \leq \delta.
\end{equation}
Moreover, keeping in mind that $e \in \partial \Om_\sigma$, and setting $u_{\sigma}(\cdot) = u(\delta_{\sigma} \cdot)$, the boundary value  problem \eqref{bp} becomes
\begin{equation}\label{bp1}
\sum a_{ij} (\delta_{\sigma} p) X_i X_j u_{\sigma}(p) = \sigma^2 f(\delta_{\sigma} p) \ \ \text{in}\ \ \mathscr V(\sigma^{-1}),\ \ \ \ \ \ \ u=0\ \ \text{on}\ \ \mathscr S(\sigma^{-1}).
\end{equation}
Therefore, with the hypothesis \eqref{smn} in place, after substituting $\Om$ with $\Omega_{\sigma}$, $u$ with $u_\sigma$ and so on, we can assume that \eqref{sm0} and \eqref{sm1} hold, where $\delta$ will be chosen later.
Our proof will be divided into the following three steps:
\begin{itemize}
\item[1.] existence of the $k$-th order Taylor polynomial at every boundary point of $\mathscr S_{1/2}$ via compactness arguments;
\item[2.] H\"older continuity of the $k$-th order derivatives at the boundary;
\item[3.] H\"older continuity of the $k$-th order derivatives up to the boundary.
\end{itemize}
Inspired by ideas in \cite{DS1, DS2}, we next introduce the relevant notion of approximating polynomial.
\begin{definition}
We say that a  polynomial $P$ of  degree less than or equal to  $k-1$ is a \emph{$k$-th order approximating polynomial}  for $\Ls u= f$ at $e$ if for all $I$ such that $|I| \leq k-2$ one has
\[
X^I\left(\Ls(d P)\right)(e)=X^I f(e).
\]
\end{definition}
Let $P$ be a given stratified polynomial of the type $\sum_{J} a_J z^J$. Then  one has
\[
\Ls (dP)=P \Ls d+d\Ls P+2\langle A(p) \nabla_H d,\nabla_H P\rangle.
\] 
Since $\Ls x_m = 0$, in view of \eqref{distres} we have
\begin{eqnarray*}
P\Ls d+d\Ls P&=& |\nabla_{H}d(e)|x_m\Ls P+(P_d+d_0)\Ls P+P\Ls(P_d+d_0).
\end{eqnarray*}
By adding and subtracting the identity matrix to $A(p)$, we obtain
\begin{equation}
\label{xmLP}
 |\nabla_{H} d(e)|x_m\Ls P= |\nabla_{H} d(e)|x_m\Delta_{H}P+ |\nabla_{H} d(e)|x_m(\Ls-\Delta_{H}) P,
\end{equation}
and also
\begin{align}
\label{aijXdXP}
2\langle A(p) \nabla_H d,\nabla_H P\rangle & = 2 \langle\nabla_{H}\left(|\nabla_{H} d(e)|x_m+P_d+d_0\right),\nabla_{H}P\rangle
\\
& +2\langle(A(p)-I)\nabla_H d,\nabla_H P\rangle
\notag\\
& = 2|\nabla_{H} d(e)|X_mP + 2 \langle\nabla_{H}\left(P_d+d_0\right),\nabla_{H}P\rangle
\notag\\
& +2\langle(A(p)-I)\nabla_H d,\nabla_H P\rangle,
\notag
\end{align}
where in the last equality we have used that $\nabla_{H}x_m = e_m$. We now recall that, for $j\in\{1,...,r\}$ and $s\in \{1,...,m_j\}$, we denote by $\overline{(j,s)}$ a multi-index which has 1 in the $(j,s)$ position and zeroes everywhere else, with the agreement that, when $j=1$, we indicate $\overline{(1, i)}$ with $\overline{i}$. 
If we then assume that, for a given multi-index $J\in (\mathbb N \cup \{0\})^N$ of the type  $J= (\beta_{1,1},\ldots,\beta_{1,m},\ldots,\beta_{r,1},\ldots,\beta_{r,m_r})$  (see \eqref{I} above)  with $d(J)\leq k-1$, we take $P(z)=z^J$,  then \eqref{xiex2} gives
\begin{eqnarray}
\label{XmP}
X_mP&=& \beta_{1,m}z^{J-\bar{m}}+\sum^r_{j=2}\sum^{m_j}_{s=1}\left(\sum_{0\leq d(I^{(j-1)})= j-1}b^s_{m,I^{(j-1)}}\beta_{j,s}z^{I^{(j-1)}+J-\overline{(j,s)}}\right)
\end{eqnarray}
Using the stratified Taylor's inequality  \eqref{stp1}, as well as \eqref{smn} and  \eqref{sm1}, it follows that we may write
\begin{align}
\label{l.o.t.}
 &|\nabla_{H} d(e)|\left(x_m(\Ls-\Delta_{H}) P+2\sum^m_{i,j}(a_{i,j}-\delta_{ij})X_idX_jP\right)\\
& +  (P_d+d_0)\Ls P+P\Ls(P_d+d_0)+2\langle\nabla_{H}\left(P_d+d_0\right),\nabla_{H}P\rangle\notag \\
= &\sum_{I}c^J_Iz^I+w_J\notag,
\end{align}
where $c_{I}^J= O(\delta)$,   $c_{I}^J \neq 0$ provided $d(J) \leq d(I) \leq k-2$, and $w_J= O(\delta |p|^{k-2+\alpha})$. 
Therefore, using \eqref{xmLP}-\eqref{l.o.t.}, and  computing $\Delta_{H} P$ using the expression in \eqref{subLap}, we find 
\begin{align}\label{ex1}
& \Ls(dP)  = |\nabla_{H} d(e)| \beta_{1,m}(\beta_{1,m} +1) z^{J- \bar{m}}  + |\nabla_{H} d(e)|\sum_{i=1}^{m-1}  \beta_{1, i} (\beta_{1, i} -1) z^{J-2\bar{i} +\bar{m}} \\
  &+ |\nabla_{H} d(e)| \sum^{m}_{i=1}\sum^r_{j=2}\sum^{m_j}_{s=1}\left(\sum_{d(I^{(j-1)})= j-1}\alpha_{1,i} \beta_{j,s} b^s_{i,I^{(j-1)}}z^{I^{(j-1)- \overline{i} +J - \overline{(j,s)} + \overline{m}}}\right)\notag\\& +2|\nabla_{H} d(e)|\sum^r_{j=2}\sum^{m_j}_{s=1}\left(\sum_{ d(I^{(j-1)})= j-1}b^s_{m,I^{(j-1)}}\beta_{j,s}z^{I^{(j-1)}+J-\overline{(j,s)}}\right)\notag\\
& +2|\nabla_{H} d(e)|\sum^{m}_{i=1}\beta_{1,i}\sum^r_{j=2}\sum^{m_j}_{s=1}\left(\sum_{d(I^{(j-1)})= j-1}b^s_{i,I^{(j-1)}}\beta_{j,s}z^{I^{(j-1)}+J-\overline{(j,s)}-\bar{i}+\bar{m}}\right)\notag \\
& +|\nabla_{H} d(e)|\sum^{m}_{i=1}\sum^r_{j,j'=2}\sum^{m_j}_{s,s'=1}\left(\sum_{\substack{d(I^{(j-1)})= j-1\\  d(I^{(j'-1)})= j'-1 }}\beta_{j,s}\beta_{j',s'}b^s_{i,I^{(j-1)}}b^{s'}_{i,I^{(j'-1)}}z^{I^{(j-1)}+I^{(j'-1)}+J-\overline{(j,s)}-\overline{(j',s')}+\bar{m}}\right)\notag\\
&  +|\nabla_{H} d(e)|\sum^{m}_{i=1}\sum^r_{\substack{j,j'=2\\ j'\leq j-1}}\sum^{m_j}_{s,s'=1}\left(\sum_{\substack{d(I^{(j-1)})= j-1\\  d(I^{(j'-1)})= j'-1 }}\alpha_{j',s'}\beta_{j,s}b^s_{i,I^{(j-1)}}b^{s'}_{i,I^{(j'-1)}}z^{I^{(j-1)}-\overline{(j',s')}+I^{(j'-1)}+J-\overline{(j,s)}+\bar{m}}\right)\notag\\
& +\sum_{I} c_{I}^J z^{I}  +w_J(p)\notag.
\end{align}
Moreover, in view of \eqref{smn}, \eqref{sm0} and \eqref{sm1}, we have that $|c_{I}^J|, ||w_J||_{\Gamma^{k-2+\alpha}} \leq C\delta$. In general, if $P$ is of the form 
\[
P(z)=\sum_{0\leq d(J)\leq k-1}b_Jz^J,
\]
then  from \eqref{ex1} it follows that
\begin{equation}\label{apro}
\Ls (dP) = R(z) + w,
\end{equation}
with $w$ such that  
\begin{equation}\label{w}
||w||_{\Gamma^{k-2+\alpha}} \leq C\delta,\ \ \  |w(p)| = O(\delta|p|^{k-2+\alpha}),
\end{equation}
and 
\begin{equation}\label{R}
R= \sum_{\{J\mid d(J) \leq k-2\}}  e_J z^J,
\end{equation}
where the coefficients $e_J$ satisfy the following algebraic condition
\begin{align}\label{alg}
 e_J =& |\nabla_{H} d(e)| (\beta_{1,m} + 1) (\beta_{1,m}+2) b_{J+\bar{m}}  +|\nabla_{H} d(e)|\sum_{i \neq m}  (\beta_{1,i} + 1) (\beta_{1,i}+2) b_{J+2\bar{i} - \bar{m}}\\
  &+ |\nabla_{H} d(e)| \sum^{m}_{i=1}\sum^r_{j=2}\sum^{m_j}_{s=1}\left(\sum_{d(I^{(j-1)})= j-1}a_1(J, i, j, s) b_{J- I^{(j-1)} +\overline{(j,s)} + \bar{i} - \bar{m}}\right)\notag\\ 
 &+2|\nabla_{H} d(e)|\sum^r_{j=2}\sum^{m_j}_{s=1}\left(\sum_{ \substack{d(I^{(j-1)})= j-1}} a_2 (J,m, i, j, s) b_{J +\overline{(j,s)}-I^{(j-1)}}\right)\notag \\
 & +2|\nabla_{H} d(e)|\sum^{m}_{i=1}\sum^r_{j=2}\sum^{m_j}_{s=1}\left(\sum_{\substack{d(I^{(j-1)})= j-1}} a_3(i, J,s)b_{J+\overline{(j,s)}+\bar{i}-\bar{m}-I^{(j-1)}}\right)\notag \\
 &+|\nabla_{H} d(e)|\sum^{m}_{i=1}\sum^r_{j,j'=2}\sum^{m_j}_{s,s'=1}\left(\sum_{\substack{d(I^{(j-1)})= j-1\\ d(I^{(j'-1)})= j'-1}} a_4(J,j,s,j',s',i)b_{J+\overline{(j,s)}+\overline{(j',s')}-\bar{m}-I^{(j-1)}-I^{(j'-1)}}\right)\notag\\
 &+|\nabla_{H} d(e)|\sum^{m}_{i=1}\sum^r_{\substack{j,j'=2\\ j'\leq j-1}}\sum^{m_j}_{s,s'=1}\left(\sum_{\substack{d(I^{(j-1)})= j-1\\  d(I^{(j'-1)})= j'-1}}a_5(J,j,s,j',s',i)b_{J+\overline{(j,s)}+\overline{(j',s')}-\bar{m}-I^{(j-1)}-I^{(j'-1)}}\right)\notag\\
& +\sum_I c_{J}^I b_I. \notag
\end{align}
We now introduce the following.
\newline
\textbf{\underline{Algorithm}:}\label{A}
Given a stratified  polynomial $R= \sum_{J: d(J) \leq k-2}  e_J z^J$,  our objective is finding a polynomial of degree at most $k-1$, i.e., a polynomial of the form $P= \sum_{J: d(J) \leq k-1}  b_J z^J$, such that \eqref{alg} holds. 
This will be possible if, given constants $e_J$, we can solve the linear system \eqref{alg} above for the unknowns $b_I$. Notice that \eqref{alg} is an equation where $b_{J+\bar{m}}$ is a linear combination  of  $b_I$, such that either 
\begin{itemize} 
\item $d(I)<d(J)+1$ or 
\item$d(I)=d(J)+1$, then $\beta^I_{1,m} < \beta^J_{1,m}+1$. 
\end{itemize}
Since this is always the case, it follows that we can solve \eqref{alg} by arbitrarily assigning all the coefficients $b_I$ when $I$  is a multi-index having $ \beta^I_{1,m}= 0$.

We next proceed with the proof of \emph{Step 1}.

\textbf{Step 1:}
In this step, we show the existence of a $k-$th order Taylor polynomial at every boundary point.
The next compactness lemma is crucial for the proof of $\Gamma^{k,\alpha}$ decay estimate at the boundary. 
\begin{lemma}[Compactness Lemma]
\label{MCL}
Let $u\in \Gamma^2(\V_1)\cap C(\overline{\V_1})$ be a solution of \eqref{dp0}. Assume that for some small $0<r\leq 1$ there exists an approximating polynomial $P= \sum_{d(I) \leq k-1} b_I z^I$, with  $|b_I|\leq  1$, such that 
\begin{equation}
\label{udP1}
\|u-dP\|_{L^{\infty}(\Omega\cap B(r))}\leq r^{k+\alpha}.
\end{equation}
Then, if $\delta>0$ is small enough in \eqref{sm0}, there exists a universal $0<\rho<1$ and $\widetilde{P}$ approximating  polynomial of degree at most  $k-1$ such that
\[
\|u-d\widetilde{P}\|_{L^{\infty}(\Omega\cap B(\rho r))}\leq (\rho r)^{k+\alpha},
\]
and
\[
\|P-\widetilde{P}\|_{L^{\infty}(\Omega\cap B( r))}\leq Cr^{k-1+\alpha}.
\]
\end{lemma}
\begin{proof} Let 
\begin{equation}
v(p)\overset{def}{=} \frac{(u-(dP))\circ(\delta_{r}(p))}{r^{k+\alpha}},\quad p\in \widetilde{\Omega}\cap B(1),
\end{equation}
where $\widetilde{\Omega}=\Omega_{r}$. We immediately observe that the validity of \eqref{udP1} implies  that
\begin{equation}
\label{vbdd1}
\|v\|_{L^{\infty}(\widetilde{\Omega}\cap B(1))}\leq 1.
\end{equation}
Furthermore, we obtain
\begin{equation}\label{newl}
\begin{cases}
\tilde \Ls v \overset{def}= \sum a_{ij} (\delta_r \cdot ) X_i X_j v= r^{-k+2-\alpha}\left(f(\delta_r(p))-\Ls(dP)(\delta_r(p))\right) \quad &\text{in $\widetilde{\Omega}\cap B(1)$},\\
v=0\quad &\text{on $\partial\widetilde{\Omega}\cap B(1)$}.
\end{cases}
\end{equation}
Since $f\in \Gamma^{k-2,\alpha}$, we have that
\begin{equation}
\label{Tayf}
|f(p)-P^f_e(p)|\leq C\delta d(p,e)^{k-2+\alpha},
\end{equation}
where $P^f_e$ is the Taylor  polynomial of $f$ at $e$  of homogeneous degree at most $k-2$. We note that \eqref{Tayf} is a consequence of the stratified Taylor inequality \eqref{stp1} combined with the reduction in \eqref{smn} that guarantees $||f||_{\Tau^{k-2, \alpha}}\le \delta$. Recall that $P$ is an approximating polynomial with respect to $\Ls$  and $f$, hence in \eqref{apro} we must have  
\begin{equation}\label{PfR}
P^f_e=R.
\end{equation}
To clarify \eqref{PfR}, we note first that from \eqref{apro} it follows that $R$ is the $(k-2)$-th order Taylor polynomial  of $\Ls (dP)$ at $e$ because the remainder $w$ in \eqref{apro} is $O(\delta|p|^{k-2+\alpha})$, see \eqref{w}. Furthermore,  by the definition of our approximating polynomial, we have that $X^I (\Ls (dP)(e)= X^{I} f(e)$ for all $|I| \leq k-2$, and thus (by the definition of the Taylor polynomial) both $f$ and $\Ls (dP)$ have the same Taylor polynomial of order $k-2$ at $e$. Noting that $P^f_e$ is the Taylor polynomial of $f$ at $e$ of order $(k-2)$, this ensures that \eqref{PfR} hold. Continuing further,  if we combine \eqref{apro} (in which $R$ is now replaced by $P^f_e$) with \eqref{Tayf}, we find that that  term in the right-hand side  of \eqref{newl} can estimated as follows
\begin{align*}
& \ \ \  \left| r^{-k+2-\alpha}\left(f(\delta_r(p))-\Ls(dP)(\delta_r(p))\right) \right|
\\ 
& = \left| r^{-k+2-\alpha}\left(f(\delta_r(p))- P^f_e(\delta_r(p))- w(\delta_r(p))\right) \right|
\\
& \leq r^{-k+2-\alpha} \left( |f(\delta_r(p))- P^f_e(\delta_r(p))| + |w(\delta_r(p))| \right)\leq C \delta,
\end{align*}
where in the last inequality we have used \eqref{Tayf} and \eqref{w}.
This ensures the validity of the following estimate
\begin{equation}
\label{bdLv}
|\tilde \Ls v|\leq C\delta. 
\end{equation}
If $\delta>0$ is sufficiently small, Proposition \ref{Lip} implies the existence of a constant $C > 0$, depending exclusively on $\lambda$ but not on $\delta$,  such that
\[
\|v\|_{\Gamma^{0,1}(\widetilde{\Omega}\cap B(4/5)))}\leq C.
\]
Similarly to what we did in \cite{BGM}, in the logarithmic coordinates we now extend  $v$ to a function $V: B(4/5)\to \R$  such that $V\equiv v$ in $\widetilde{\Omega}\cap B(4/5)$ (we use the classical procedure in e.g. \cite[p. 14]{LM}):
\begin{equation}\label{ext}
V(x',\tilde{x}_m,y)=
\begin{cases}
v(x',\tilde{x}_m,y)\ \ \ \ \ \ \ \ \ \ \ \ \ \ \ \ \tilde x_m \overset{def}= x_m - h(x', y) \ge 0,
\\
\sum^{k+2}_{i=1}c_iv\left(x',-\frac{\tilde{x}_m}{i},y\right)\ \quad \tilde{x}_m<0,
\end{cases}
\end{equation}
where the constants $c_i$'s are determined by the system of equations
\begin{equation}
\label{constants}
\sum^{k+2}_{i=1}c_i(-1/i)^m=1,\quad m=0,1,2,..., k, k+1.
\end{equation}
It follows that
\begin{equation}\label{ext}
\|V\|_{\Tau^{{0,1}}(B(4/5))}\leq C'\ \|v\|_{\Tau^{{0,1}}(\widetilde{\Omega}\cap B(4/5)))}\leq C_1,
\end{equation}
for some $C_1 = C_1(\lambda,\alpha)>0$. 
By a compactness argument analogous to that in the proof of \cite[Lemma 4.1]{BGM}, using  the equicontinuous Lipschitz estimate in Proposition \ref{Lip}, we can assert that  given $\varepsilon>0$ there exists a sufficiently small $\delta>0$  in \eqref{sm0}, for which there is a $v_0$ such that
\begin{equation}\label{lim}
\begin{cases}
\Delta_{H}v_0=0 \quad &\text{in  $B(4/5)\cap \{x_m>0\}$},\\
v_0=0\quad &\text{on $ B(4/5)\cap \{x_m=0\}$},
\end{cases}
\end{equation}
with 
\[
\|v-v_0\|_{L^{\infty}(\widetilde{\Omega}\cap B(1/2))}\leq \varepsilon.
\]
Said differently, when $\Om \cap B(1)$ (and consequently $\tilde \Om \cap B(1)$) is flat enough, we can approximate $v$ by a solution to a ``flat" boundary value problem.
Moreover, it follows from \eqref{ext} that  in the region $\{x_m <0\}$,  $v_0$  satisfies
\begin{equation}\label{v0}
v_0(x', x_m, y) = \sum_{i=1}^{k+2} c_i v\left(x', -\frac{x_m}{i}, y\right),
\end{equation}
where $c_i$'s satisfy \eqref{constants}. 
From the   boundary regularity in Theorem \ref{KN}, and the fact that \eqref{v0} determines a $C^{k+1}$ extension across $\{x_m=0\}$, we find
\[
\|v_0\|_{C^{k+1}(B(1/2))}\leq C.
\]
Therefore, there exists a   polynomial  $Q$, with homogeneous degree less than or equal to $k-1$, such that 
\begin{equation}
\label{Polv0}
\|v_0-x_mQ\|_{L^{\infty}(B(\rho))}\leq C_0 \rho^{k+1},\quad \text{with $0<\rho<\frac{1}{2}$}.
\end{equation}
Given $0<\alpha<1$, we now fix $\rho>0$ such that 
\begin{equation}
\label{rhofix}
C_0\rho^{k+1}=\frac{\rho^{k+\alpha}}{4}.
\end{equation}
Then, we take 
\begin{equation}
\label{epsilonfix}
\varepsilon\overset{def}{=}\frac{\rho^{k+\alpha}}{16},
\end{equation}
which in turn decides the choice of $\delta$.
We may thus conclude that 
\begin{equation}\label{fr1}
\|v-x_mQ\|_{L^{\infty}(\widetilde{\Omega}\cap B(\rho))}\leq \left(\frac{1}{4}+\frac{1}{16}\right)\rho^{k+\alpha}.
\end{equation}
Since
$$
\tilde d (\cdot) \overset{def}= \frac{d(\delta_r \cdot)}{r |\nabla_{H} d(e)|} = x_m + O(\delta),
$$
for a possibly smaller $\delta$, we obtain from \eqref{fr1} that the following holds
\begin{equation}\label{ki}
\|v- \tilde d Q\|_{L^{\infty}(\widetilde{\Omega}\cap B(\rho))}\leq \left(\frac{3}{4} \right)\rho^{k+\alpha}.\end{equation}
Rescaling back to $u$, we find 
 \begin{align}
 \label{udPQ}
\Bigg\|u-d\left(P+\frac{r^{k-1+\alpha}}{|\nabla_{H}d(e)|}Q\circ \delta_{r^{-1}}\right)\Bigg\|_{L^{\infty}(\Omega\cap B(\rho r))}\leq\left(\frac{3}{4}\right)(r\rho)^{k+\alpha}. 
\end{align}
We emphasise here that 
\[
P+\frac{r^{k-1+\alpha}}{|\nabla_{H}d(e)|}Q\circ \delta_{r^{-1}}
\]
 might not be an approximating polynomial for $u$. Hence,  we need to further modify $Q$ to some other polynomial $\widetilde{Q}$, without essentially
modifying the estimate \eqref{udPQ}.
With this in mind, we write 
\[
Q=\sum_{0\leq d(I)\leq k-1}b_Iz^I\quad \text{and}\quad \widetilde{Q}=\sum_{0\leq d(I)\leq k-1}\widetilde{b}_Iz^I.
\]
In view of \eqref{alg}, after a computation, we find that for $P+ r^{k-1+\alpha} \tilde Q \circ \delta_{r^{-1}}$ to be approximating,   the coefficients $\widetilde{b}_I$ have to satisfy the following algebraic conditions:

\begin{align}\label{alg1}
 &0=|\nabla_{H} d(e)| (\beta_{1,m} + 1) (\beta_{1,m}+2) \widetilde{b}_{J+\bar{m}}  +|\nabla_{H} d(e)| \sum_{i \neq m}  (\beta_{1,i} + 1) (\beta_{1,i}+2) \widetilde{b}_{J+2\bar{i} - \bar{m}}\\
  &+ |\nabla_{H} d(e)| \sum^{m}_{i=1}\sum^r_{j=2}\sum^{m_j}_{s=1}\left(\sum_{d(I^{(j-1)})= j-1}a_1(J, i, j, s) \widetilde{b}_{J- I^{(j-1)} +\overline{(j,s)} + \bar{i} - \bar{m}}\right)\notag\\ 
 &+2|\nabla_{H} d(e)|\sum^r_{j=2}\sum^{m_j}_{s=1}\left(\sum_{ \substack{d(I^{(j-1)})= j-1}} a_2 (J,m, i, j, s) \widetilde{b}_{J +\overline{(j,s)}-I^{(j-1)}}\right)\notag \\
 & +2|\nabla_{H} d(e)|\sum^{m}_{i=1}\sum^r_{j=2}\sum^{m_j}_{s=1}\left(\sum_{\substack{d(I^{(j-1)})= j-1}} a_3(i, J,s)\widetilde{b}_{J+\overline{(j,s)}+\bar{i}-\bar{m}-I^{(j-1)}}\right)\notag \\
 &+|\nabla_{H} d(e)|\sum^{m}_{i=1}\sum^r_{j,j'=2}\sum^{m_j}_{s,s'=1}\left(\sum_{\substack{d(I^{(j-1)})= j-1\\ d(I^{(j'-1)})= j'-1}} a_4(J,j,s,j',s',i)\widetilde{b}_{J+\overline{(j,s)}+\overline{(j',s')}-\bar{m}-I^{(j-1)}-I^{(j'-1)}}\right)\notag\\
 &+|\nabla_{H} d(e)|\sum^{m}_{i=1}\sum^r_{\substack{j,j'=2\\ j'\leq j-1}}\sum^{m_j}_{s,s'=1}\left(\sum_{\substack{d(I^{(j-1)})= j-1\\  d(I^{(j'-1)})= j'-1}}a_5(J,j,s,j',s',i)\widetilde{b}_{J+\overline{(j,s)}+\overline{(j',s')}-\bar{m}-I^{(j-1)}-I^{(j'-1)}}\right)\notag\\
& +\sum_I \widetilde{c}_{J}^I \widetilde{b}_I,\notag
\end{align}

where \begin{equation}\label{smo0} \widetilde{c}_{J}^I= r^{d(J) +1 -d(I)} c_{J}^I.\end{equation}

Moreover, since $c_{J}^I \neq 0$ provided that $d(I) \leq d(J)$ and $r \leq 1$, it follows  from \eqref{smo0} that
\begin{equation}\label{smo}
|\widetilde{c}_{J}^I| \leq C \delta.
\end{equation}
Since $x_mQ$ is $\Delta_{H}-$harmonic, we find that the  coefficients $b_I$ must satisfy
\begin{align}\label{alg2}
 &0= (\beta_{1,m} + 1) (\beta_{1,m}+2) b_{J+\bar{m}}  +\sum_{i \neq m}  (\beta_{1,i} + 1) (\beta_{1,i}+2) b_{J+2\bar{i} - \bar{m}}\\
  &+  \sum^{m}_{i=1}\sum^r_{j=2}\sum^{m_j}_{s=1}\left(\sum_{d(I^{(j-1)})= j-1}a_1(J, i, j, s) b_{J- I^{(j-1)} +\overline{(j,s)} + \bar{i} - \bar{m}}\right)\notag\\ 
 &+2\sum^r_{j=2}\sum^{m_j}_{s=1}\left(\sum_{ \substack{d(I^{(j-1)})= j-1}} a_2 (J,m, i, j, s) b_{J +\overline{(j,s)}-I^{(j-1)}}\right)\notag \\
 & +2\sum^{m}_{i=1}\sum^r_{j=2}\sum^{m_j}_{s=1}\left(\sum_{\substack{d(I^{(j-1)})= j-1}} a_3(i, J,s) b_{J+\overline{(j,s)}+\bar{i}-\bar{m}-I^{(j-1)}}\right)\notag \\
 &+\sum^{m}_{i=1}\sum^r_{j,j'=2}\sum^{m_j}_{s,s'=1}\left(\sum_{\substack{d(I^{(j-1)})= j-1\\ d(I^{(j'-1)})= j'-1}} a_4(J,j,s,j',s',i)b_{J+\overline{(j,s)}+\overline{(j',s')}-\bar{m}-I^{(j-1)}-I^{(j'-1)}}\right)\notag\\
 &+\sum^{m}_{i=1}\sum^r_{\substack{j,j'=2\\ j'\leq j-1}}\sum^{m_j}_{s,s'=1}\left(\sum_{\substack{d(I^{(j-1)})= j-1\\  d(I^{(j'-1)})= j'-1}}a_5(J,j,s,j',s',i)b_{J+\overline{(j,s)}+\overline{(j',s')}-\bar{m}-I^{(j-1)}-I^{(j'-1)}}\right).\notag
\end{align}
From \eqref{alg1} and \eqref{alg2} we find that $d_J = \widetilde{b}_J - \frac{b_J}{|\nabla_{H} d(e) |}$ should satisfy the following  system 
\begin{align}\label{alg4}
 &-\sum_I  \widetilde{c}_{J}^I \frac{b_I}{|\nabla_{H} d(e)|}=|\nabla_{H} d(e)| (\beta_{1,m} + 1) (\beta_{1,m}+2) d_{J+\bar{m}}  +|\nabla_{H} d(e)| \sum_{i \neq m}  (\beta_{1,i} + 1) (\beta_{1,i}+2) d_{J+2\bar{i} - \bar{m}}\\
  &+ |\nabla_{H} d(e)| \sum^{m}_{i=1}\sum^r_{j=2}\sum^{m_j}_{s=1}\left(\sum_{d(I^{(j-1)})= j-1}a_1(J, i, j, s) d_{J- I^{(j-1)} +\overline{(j,s)} + \bar{i} - \bar{m}}\right)\notag\\ 
 &+2|\nabla_{H} d(e)|\sum^r_{j=2}\sum^{m_j}_{s=1}\left(\sum_{ \substack{d(I^{(j-1)})= j-1}} a_2 (J,m, i, j, s) d_{J +\overline{(j,s)}-I^{(j-1)}}\right)\notag \\
 & +2|\nabla_{H} d(e)|\sum^{m}_{i=1}\sum^r_{j=2}\sum^{m_j}_{s=1}\left(\sum_{\substack{d(I^{(j-1)})= j-1}} a_3(i, J,s)d_{J+\overline{(j,s)}+\bar{i}-\bar{m}-I^{(j-1)}}\right)\notag \\
 &+|\nabla_{H} d(e)|\sum^{m}_{i=1}\sum^r_{j,j'=2}\sum^{m_j}_{s,s'=1}\left(\sum_{\substack{d(I^{(j-1)})= j-1\\ d(I^{(j'-1)})= j'-1}} a_4(J,j,s,j',s',i)d_{J+\overline{(j,s)}+\overline{(j',s')}-\bar{m}-I^{(j-1)}-I^{(j'-1)}}\right)\notag\\
 &+|\nabla_{H} d(e)|\sum^{m}_{i=1}\sum^r_{\substack{j,j'=2\\ j'\leq j-1}}\sum^{m_j}_{s,s'=1}\left(\sum_{\substack{d(I^{(j-1)})= j-1\\  d(I^{(j'-1)})= j'-1}}a_5(J,j,s,j',s',i)d_{J+\overline{(j,s)}+\overline{(j',s')}-\bar{m}-I^{(j-1)}-I^{(j'-1)}}\right)\notag\\
& +\sum_I \widetilde{c}_{J}^I d_I.\notag
\end{align}
From \eqref{smo} it follows that the left-hand side of the linear system in \eqref{alg4} is bounded by $C\delta$. Therefore, using the above \textbf{Algorithm}, we can determine $d_I$'s such that $|d_I| \leq C\delta$. It thus  follows that for a new constant $C>0$ we have
\[
|| |\nabla_{H} d(e)| \tilde Q - Q||_{L^{\infty}(\tilde \Om \cap B(\rho))} \leq C\delta.
\]
Using this inequality in the estimate \eqref{ki}, we obtain
\begin{equation}\label{ki1}
\|v- \tilde d |\nabla_{H} d(e)|  \tilde Q\|_{L^{\infty}(\widetilde{\Omega}\cap B(\rho))}\leq \left(\frac{3}{4} \right)\rho^{k+\alpha} + C\delta  \leq \rho^{k+\alpha},\end{equation}
provided $\delta$ is chosen smaller if necessary.
Rescaling back to $u$, we obtain 
 \begin{align}
 \label{newu}
\Bigg\|u-d\left(P+r^{k-1+\alpha}\tilde Q\circ \delta_{r^{-1}}\right)\Bigg\|_{L^{\infty}(\Omega\cap B(\rho r))}\leq(r\rho)^{k+\alpha}. 
\end{align}
The conclusion of the lemma thus follows with
\[
\tilde P=P+r^{k-1+\alpha}\tilde Q\circ \delta_{r^{-1}}.\]

\end{proof}

Using Lemma \ref{MCL}, we next show the existence of the $k$-th order Taylor polynomial at $e \in \partial \Om$.

\begin{lemma}\label{btaylor}
Let $u$ be as in Lemma \ref{MCL}. Then there exists a polynomial $P_e$ at $e$ of homogeneous degree less than or equal to $k$  such that for all $p \in \Om \cap B(1/2)$
\begin{equation}\label{bdry}
|u(p)- P_e(p)| \leq C |p|^{k+\alpha}.
\end{equation}
\end{lemma}
\begin{proof}
Using the above \textbf{Algorithm}, we  first find an approximating polynomial $P_1$. Then by multiplying $u$ and $P_1$ with a small constant, we may assume that the hypothesis of the Lemma \ref{MCL} is satisfied for $r=1/2$. Iterating Lemma \ref{MCL} at  scales $\rho/2, \rho^2/2$, and so on, we obtain a limiting  approximating polynomial  $P_0$ such that
\begin{equation}\label{bdry1}
|u(p)- dP_0(p)| \leq C |p|^{k+\alpha},
\end{equation}
for all $p \in \Om \cap B(1/2)$. Since $d P_0 \in \Gamma^{k, \alpha}(\overline{\Om \cap B(1)})$,  by letting $P_e$ to be the $k$-th order Taylor polynomial of $dP_0$ at $e$, we find that \eqref{bdry} follows.

\end{proof}

\begin{remark}\label{dec1}
From the proof of Lemma \ref{btaylor} it follows that since $P_e$ equals the $k$-th order Taylor polynomial of $dP_0$, where $P_0$ is approximating, one has for all $I$ such that $|I| \leq k-2$
\begin{equation}
X^{I} ( \mathscr LP_e) (e) = X^I f(e).
\end{equation}
Then from the stratified Taylor inequality \eqref{stp1} we have
\begin{equation}\label{dec2}
|X^I (\mathscr LP_e) (p) - X^I f(p) | \leq C |p|^{k-2 - |I| +\alpha}
\end{equation}
for all $I$ such that $|I| \leq k-2$.  Furthermore, by left translation it follows that for every $g \in \partial \Om \cap B(1/2)$, there exists a polynomial $P_g$ of degree at most $k$ such that
\begin{equation}\label{dec0}
\begin{cases}
|u(g \circ p) - P_g(p)| \leq C|p|^{k+\alpha},
\\
|X^I (\sum_{i,j=1}^m a_{ij}(g \circ \cdot) X_iX_j P_g)(p) - X^I f( g \circ p) | \leq  C|p|^{k-2 - |I| + \alpha}, \ |I| \leq k-2,\\
\operatorname{sup}_{|I|=k-2} [X^I (\sum_{i,j=1}^m a_{ij}(g \circ \cdot) X_i X_jP_g)(\cdot) - X^I f( g \circ \cdot)]_{\Gamma^{0, \alpha}} \leq C.
\end{cases}
\end{equation}

\end{remark}

This completes the proof of \emph{Step 1}. We now proceed with that of \emph{Step 2}.

\textbf{Step 2:} In our subsequent discussion given any boundary point $ p\in \mathscr S_{1/2}$ we will denote by $P_{p}$ the corresponding Taylor polynomial of $u$ at $p$ whose existence has been established in Lemma \ref{btaylor}. Our main objective in this step is to show that,   given $p_1, p_2 \in \mathscr S_{1/2}$, the following estimate holds for some $K_0>0$ universal: 
\begin{equation}\label{bdcmp}
\operatorname{sup}_{|J|=k} |X^J P_{p_1} (e) - X^J P_{p_2}(e)| \leq  K_0\ d(p_1,p_2)^{\alpha}.
\end{equation}

Let $t= d(p_1,p_2)$. We consider a ``non-tangential" point $p_3 \in \V_1$ at a (pseudo-)distance from $p_1$ comparable to $t$,  i.e., let $p_3$ be such that
\begin{equation}\label{nont}
d(p_3, p_1) \sim t,\ d(p_3, \pa \Om)  \sim t,
\end{equation}
where we have let $d(p,\pa \Om)= \underset{p' \in \pa \Om}{\inf} \ d(p,p')$.
Since $\mathscr S_1$ is a non-characteristic $C^{1,\alpha}$ portion of $\pa \Om$, it is possible to find such a point $p_3$. Arguing as in the proof of Theorem 7.6 in \cite{DGP}, at any scale $t$ one can find a non-tangential pseudo-ball from inside centered at $p_3$. The non-tangentiality of $B(p_3,at)$ means that there exists a universal $a>0$ sufficiently small (which can be seen to depend on the Lipschitz character of $\pa \Om$ near the non-characteristic portion $\mathscr S_1$), such that for some $c_0$ universal one has for all $p \in B(p_3,a t)$ 
\[
d(p,\pa \Om) \geq  c_0 t.
\]
For $s=1,2$, we have that in $B(p_3, at)$, $v^s= u( \cdot) - P_{p_s}( p_s^{-1} \circ \cdot)$, solves
\begin{equation}\label{tu}
\sum_{ij} a_{ij} X_i X_j v^s = \tilde f^s \overset{def} = f - \sum_{i,j} a_{ij} X_i X_j P_{p_s} (p_s^{-1} \circ \cdot).\end{equation}
From the  estimates  in \eqref{dec0} we have for $s=1,2$,
\begin{equation}\label{p1}
\begin{cases}
||v^s||_{L^{\infty}(B(p_3, at))} \leq C t^{k+\alpha},\\
|X^I \tilde f^s (p) | \leq  Ct^{k-2 - |I| + \alpha}, \ |I| \leq k-2,\ \text{for $p \in B(p_3, at)$},\\
\sum_{|I|=k-2} [X^I \tilde f^s]_{\Gamma^{0,\alpha}} \leq C.
\end{cases}
\end{equation}
Using now the interior estimate \eqref{rei1}, \eqref{p1} and the fact that $P_{p_s} \in \mathcal{P}_k$,  we find
\begin{align}\label{rei5}
&\operatorname{sup}_{|J|=k}|X^J v^s (p_3)| =\operatorname{sup}_{|J|=k} |X^J u (p_3) - X^J P_{p_s}(e) | \\ &\leq \frac{C}{t^k} \left(||v^s||_{L^{\infty}(B(p_3, at))}+\sum_{|I| \leq k-2} t^{2+|I|} || X^I \tilde f^s||_{L^{\infty}(B(p_3, at)} + t^{k+\alpha} \sum_{|I|=k-2} [X^I \tilde f^s]_{\Gamma^{0,\alpha}} \right)
\notag
\\
& \leq C t^{\alpha},
\notag
\end{align}
where in the last inequality we have used \eqref{p1}. From \eqref{rei5} and the triangle inequality we thus have for any $J$ such that $|J|=k$, 
\begin{equation}\label{result}
|X^J P_{p_1} (e) - X^J P_{p_2}(e) | \leq |X^J u (p_3) - X^J P_{p_1}(e) | + |X^J u (p_3) - X^J P_{p_2}(e) |  \leq Ct^{\alpha},
\end{equation}
which finishes the proof of the claim \eqref{bdcmp}. 

\textbf{Step 3:} (Proof of Theorem \ref{main})
In this final step it suffices to prove that the $k$-th order horizontal derivative of $u$  is in $\Gamma^{0,\alpha}$ up to the boundary. As a first step we observe that we can find $\ve>0$ sufficiently small such that for any $p\in \V_\ve$ there exists $\bar p \in \mathscr S_{1/2}$ for which
\begin{equation}\label{eqt}
d(p,\bar p) = d(p,\pa \Om).
\end{equation}
To finish the proof of the theorem we will show that for all  $p,p' \in \V_\ve$ and $J$ such that $|J|=k$,  we have   \begin{equation}\label{fop}
|X^J u(p)- X^J u(p')|\leq  C^\star d(p,p')^{\alpha},
\end{equation}
for some  universal constant $C^\star >0$. The sought for conclusion  follows from \eqref{fop} by a standard covering argument.

Given two points  $p, p' \in \V_\ve$ we denote by $\bar p, \bar p'$  the corresponding points in $\mathscr S_{1/2}$ for which \eqref{eqt} holds.
Henceforth, we use the notation $\delta(p) = d(p,\pa \Om)$ for  $p \in \Om$. 
Without loss of generality we assume that $\delta(p)= \min \{\delta(p),\delta(p')\}$. By Lemma \ref{btaylor} and Remark \ref{dec1}  in Step 2, given $\bar p$ as in \eqref{eqt} we know that there exists a $k$-th polynomial $P_{\bar{p}}$ such that for every $q \in \V_1$ one has
\begin{equation}\label{fopb}
\begin{cases}
|u(q) - P_{\bar{p}}(\bar{p}^{(-1)} \circ q)|\leq C_2d(\bar{p},q)^{k+\alpha},\\
|X^I (\sum_{i,j=1}^m a_{ij}(\bar{p} \circ \cdot) X_i X_jP_{\bar{p}})(q) - X^I f( \bar{p} \circ q) | \leq  C|q|^{k-2 - |I| + \alpha}, \ |I| \leq k-2,\\
\operatorname{sup}_{|I|=k-2} [X^I (\sum_{i,j=1}^m a_{ij}(\bar{p} \circ \cdot)X_i X_j P_{\bar{p}})(\cdot) - X^I f( \bar{p} \circ \cdot)]_{\Gamma^{0, \alpha}} \leq C.\end{cases}
\end{equation}
There exists two possibilities:
\begin{itemize}
\item[(i)] $d(p, p') \leq \frac{\delta(p)}{2}$;
\item[(ii)] $d(p, p') > \frac{\delta(p)}{2}$.
\end{itemize}
 \textbf{Case (i)}:  We first observe that we have $B(p, \delta(p)) \subset \Omega$. Therefore, if we let $v(\cdot)=u(\cdot)-P_{\bar p}(\bar{p}^{(-1)} \circ \cdot)$, then similarly to Step 2 we see that $v$ satisfies an equation of the following type  in $B(p, \delta(p)) \subset \Omega$
\[
\sum_{ij} a_{ij} X_i X_j v(\cdot) = \tilde f,
\]
where $\tilde f$ is as in  \eqref{tu} and moreover  estimates as in \eqref{p1} hold, with $t$ replaced by $\delta(p)$. 
In particular, it   follows from \eqref{fopb} that the following estimate holds for some $\tilde C_2>0$
\begin{equation}\label{sup}
||v||_{L^{\infty}(B(p,\delta(p))} \leq \tilde{C}_2 \delta(p)^{k+\alpha}.
\end{equation}
Since $p' \in B(p, \delta(p)/2)$, using the interior regularity estimate \eqref{rei2}  in Corollary \ref{intreg1} and also  \eqref{p1}, we conclude that given $J$ such that $|J|=k$,  there exists some $\tilde C$ depending also on $\tilde{C}_2$ such that the following inequality holds
\begin{align}\label{sup2}
& |X^Jv(p) - X^J v(p')|
\\
& \leq \frac{\tilde C d(p, p')^\alpha}{\delta(p)^{k+\alpha}} \left(||v||_{L^{\infty}(B(p,\delta(p))}+\sum_{|I| \leq k-2} \delta(p)^{2+|I|} ||X^I \tilde f||_{L^{\infty}(B(p, \delta(p))}+\delta(p)^{k+\alpha} \sum_{|I|=k-2} [X^I \tilde f]_{\Gamma^{0, \alpha}(B(p, \delta(p))} \right) 
\notag\\
&\leq  \frac{\tilde C}{\delta(p)^{k+\alpha}} \delta(p)^{k+\alpha} d(p, p')^{\alpha}\ \text{(using \eqref{sup} and \eqref{p1})}\notag\\ & \leq   \tilde{C_1} d(p, p')^{\alpha}.
\notag
\end{align}
Since $P_{\bar p} \in \mathcal{P}_k$,  we have that $X^J P_{\bar p}$ is constant for any $J$ such that $|J|=k$. From the definition of $v$ it thus follows that  
\[
|X^J u(p) - X^Ju(p')| = |X^J v(p) - X^J v(p')|.
\]
From this observation and \eqref{sup2} we conclude that \eqref{fop} holds. 

\noindent \textbf{Case (ii)}: The hypothesis $d(p, p') > \frac{\delta(p)}{2}$ and \eqref{eqt} imply 
\[
d(p,\bar{p}) = d(p,\pa\Om) = \delta(p) < 2 d(p, p').
\]
Combining this observation with \eqref{pseudo}, which in particular implies a pseudo-triangle inequality,  gives  
\begin{equation}\label{n101}
d(p', \bar{p}) \leq C_0(d(p',p) + d(p,\bar{p})) \leq  C_0 (d(p',p) + 2 d(p',p)) = 3 C_0 d(p, p').
\end{equation}
Since $d(p',\bar{p}) \geq d(p',\pa \Om) = \delta(p')$, we immediately find from \eqref{n101} 
\begin{equation}\label{n103}
\delta(p') \leq 3C_0 d(p,p').
\end{equation}
From  \eqref{n101}, \eqref{n103}, and an application of the pseudo-triangle inequality, we finally have
\begin{equation}\label{dist}
 d(\bar{p}, \bar{p}') \leq C_0 (d(\bar p, p') + d(p', \bar{p}')) = C_0 (d(\bar p,p') + \delta(p')) \leq 6 C_0^2 d(p, p').
\end{equation}
Let now $a$ be the universal constant in the existence of a non-tangential (pseudo)-ball in Step 2. From the estimate \eqref{rei5} in  Step 2 we infer that the following holds for $J$ such that $|J|=k$,
\begin{equation}\label{imo2}
|X^Ju(p) - X^J P_{\bar{p}}(e)|  \leq C \delta(p)^{\alpha} \leq C d(p, p')^{\alpha}.
\end{equation}
Note that, since we are in Case (ii), in the last inequality in \eqref{imo2} we have used $\delta(p) \leq 2 d(p, p')$. 
Arguing similarly to \eqref{imo2}, we find 
\begin{equation}\label{imo1}
|X^J  u(p') - X^J P_{\bar{p}'}(e)| \leq C \delta(p')^{\alpha} \leq C d(p, p')^{\alpha},
\end{equation}
where in the last inequality in \eqref{imo1} we have this time used \eqref{n103}. 
From \eqref{bdcmp} and  \eqref{dist} we now have
\begin{equation}\label{finali}
|X^J P_{\bar p}(e) - X^J P_{\bar{p}'}(e)| \leq K_0 d(\bar p, \bar{p}')^{\alpha} \le C d(p, p')^{\alpha}.
\end{equation}
Applying the triangle inequality with the estimates \eqref{imo2}, \eqref{imo1} and \eqref{finali}, we finally conclude that \eqref{fop} holds.  This completes the proof of the theorem.


\begin{thebibliography}{99}

\bibitem[BCC]{BCC}
A. Baldi, G. Citti \& G. Cupini, \emph{Schauder estimates at the boundary for sub-Laplacians in Carnot groups},
Calc. Var. Partial Differential Equations \textbf{58}~ (2019).
\bibitem[BGM]{BGM} A. Banerjee, N. Garofalo \& I.H. Munive, \emph{Compactness methods for $\Gamma^{1,\alpha}$ boundary Schauder estimates in Carnot groups}, Calc. Var. Partial Differential Equations \textbf{58}~ (2019).
\bibitem[Be]{Be}
A. Bella\"iche, \emph{The tangent space in sub-Riemannian geometry}, in \emph{Sub-Riemannian geometry}, 1-78, Progr. Math., \textbf{144}, Birkh\"auser, Basel, 1996.

\bibitem[BU]{BU}
A. Bonfiglioli \& F. Uguzzoni, \emph{Maximum principle and propagation for intrinsically regular solutions of differential inequalities structured on vector fields}. J. Math. Anal. Appl. \textbf{322}~(2006), no. 2, 886-900. 


\bibitem[BBLU]{BBLU}
M. Bramanti, L. Brandolini, E. Lanconelli and F. Uguzzoni, \emph{Non-Divergence Equations Structured on H\"ormander Vector Fields: heat Kernels and Harnack Inequalities}, Mem. Amer. Math. Soc., Vol.  \textbf{240}~ 2010. 
\bibitem[BLU]{BLU}
A.~Bonfiglioli, E.~Lanconelli, F.~Uguzzoni,
\emph{Stratified Lie groups and potential theory for their sub-Laplacians}, Springer, Berlin 2007, xxvi+800 pp.

\bibitem[Bony]{Bony}
J.-M. Bony, \emph{Principe du maximum, in\'egalite de Harnack et unicit\'e du probl\'eme de Cauchy pour les op\'erateurs elliptiques d\'eg\'en\'er\'es}, Ann. Inst. Fourier
 \textbf{19}~(1969), 277-304.

\bibitem[Ca]{Ca}
L. Caffarelli, \emph{Interior a priori estimates for solutions of fully nonlinear equations},  Ann. of Math. (2) \textbf{130}~ (1989), no. 1, 189-213. 


\bibitem[CDG]{CDG}
L. Capogna, D. Danielli \& N. Garofalo, \emph{Subelliptic mollifiers and a basic pointwise estimate of Poincar\'e type}, Math. Z. \textbf{226}~(1997), no. 1, 147-154. 

\bibitem[CG]{CG}
L. Capogna \& N. Garofalo, \emph{Boundary behavior of nonnegative solutions of subelliptic equations in NTA domains for Carnot-Caratheodory metrics}, J. Fourier. Anal. Appl. \textbf{4}~(1998), 403-432. 

\bibitem[CGN1]{CGNajm}
L. Capogna \& N. Garofalo, \emph{ Properties of harmonic measures in the Dirichlet
problem for nilpotent Lie groups of Heisenberg type}, Amer. J. Math.
\textbf{124}, vol 2, (2002) 273-306.



\bibitem[CGN2]{CGN}
L. Capogna, N. Garofalo \& D. M. Nhieu,
\emph{Mutual absolute continuity of harmonic and surface measures for H\"ormander type operators},
Perspectives in partial differential equations, harmonic analysis and applications,
Proc. Sympos. Pure Math., \textbf{79}~(2008), 49--100.

\bibitem[CH]{CH} L. Capogna \& Q. Han.
  Harmonic Analysis at Mount Holyoke: Proceedings of an AMS-IMS-SIAM Joint Summer Research Conference on Harmonic Analysis, June 25-July 5, 2001, Mount Holyoke College, South Hadley, MA, Contemporary mathematics - American Mathematical Society, 2003, American Mathematical Society.



\bibitem[Ca]{Ca}
E. Cartan, \emph{Sur la repr\'esentation g\'eom\'etrique des
syst\`emes mat\'eriels non holonomes}, Proc. Internat. Congress
Math., vol.4, Bologna, 1928, 253-261.

\bibitem[CGS]{CGS}
G. Citti, G. Giovannardi \& Y. Sire, \emph{The method of double layer potential for the Schauder estimates at the boundary in $\mathbb H^1$}, ArXiv


  
\bibitem[CGr]{CGr}
 L. Corwin and F. P. Greenleaf, \emph{Representations of nilpotent Lie groups and their applications,
Part I: basic theory and examples},
 Cambridge Studies in Advanced Mathematics 18,
Cambridge University Press, Cambridge
(1990).

  \bibitem[D]{D} D. Danielli, \emph{Regularity at the boundary for solutions of nonlinear subelliptic equations}, Indiana J. of Math.,
1, 44 (1995), 269-286.

\bibitem[DGN]{DGN}
D. Danielli, N. Garofalo \& D.M. Nhieu, \emph{A notable family of entire intrinsic minimal graphs in the Heisenberg group which are not perimeter minimizing}. Amer. J. Math. \textbf{130}~ (2008), no. 2, 317-339. 

\bibitem[DGNP]{DGNP}
D. Danielli, N. Garofalo, D.M. Nhieu \& S.D. Pauls,  \emph{Instability of graphical strips and a positive answer to the Bernstein problem in the Heisenberg group $\mathbb H^1$}. J. Differential Geom. \textbf{81}~(2009), no. 2, 251-295.

\bibitem[DGP]{DGP}
D. Danielli, N. Garofalo  \& A. Petrosyan, \emph{The sub-elliptic obstacle problem: $C^{1, \alpha}$ regularity of the free boundary in Carnot groups of step two}, Adv. Math. \textbf{211} (2007), no. 2, 485-516.

\bibitem[DS1]{DS1} D. D. Silva, O. Savin. \emph{A note on higher regularity boundary Harnack inequality}, Discrete \& Continuous Dynamical Systems, 2015, 35 (12) : 6155-6163. doi: 10.3934/dcds.2015.35.6155

\bibitem[DS2]{DS2} D. D. Silva \& O.  Savin,\emph{ Boundary Harnack estimates in slit domains and applications to thin free boundary problems}, Rev. Mat. Iberoam. 32 (2016), no. 3, 891–912.


\bibitem[De]{De} M. Derridj, \emph{Un probl\`eme aux limites pour une classe d'op\`erateurs du second ordre hypoelliptiques}, Ann.
Inst. Fourier, Grenoble, \textbf{21}, 4 (1971), 99-148.


\bibitem[Fi1]{Fi1}
G. Fichera, \emph{Sulle equazioni differenziali lineari ellittico-paraboliche del secondo ordine}, Atti Acc. Naz. Lincei Mem. Ser. 8, \textbf{5}~(1956), 1-30.

\bibitem[Fi2]{Fi2}
G. Fichera, \emph{On a unified theory of boundary value problems for elliptic-parabolic equations of second order. Boundary value problems in differential equations}, Univ. of Wisconsin press, Madison, 1960, 97-120.




\bibitem[F]{F} 
G. Folland, Subelliptic estimates and function spaces on nilpotent Lie groups, Ark. Math., 13 (1975), 161-207.


\bibitem[FS]{FS} G. Folland \& E. M. Stein, \emph{Hardy spaces on homogeneous Carnot groups}, Mathematical Notes, \textbf{28}, Princeton University Press, Princeton, N.J., 1982.

\bibitem[G]{Gems}
N. Garofalo, \emph{Hypoelliptic operators and some aspects of analysis and geometry of sub-Riemannian spaces}, Geometry, analysis and dynamics on sub-Riemannian manifolds. Vol. 1, 123-257, EMS Ser. Lect. Math., Eur. Math. Soc., Z\"urich, 2016. 


\bibitem[GV]{GV}
N. Garofalo \& D. Vassilev, \emph{Regularity near the characteristic set in the non-linear Dirichlet problem and conformal geometry of sub-Laplacians on Carnot groups}, Math. Ann. \textbf{318}~(2000), 453-516.

\bibitem[GT]{GT}
D. Gilbarg \& N. S. Trudinger, \emph{Elliptic partial differential equations of second order}. Second edition. Grundlehren der mathematischen Wissenschaften [Fundamental Principles of Mathematical Sciences], 224. Springer-Verlag, Berlin, 1983. xiii+513 pp.

\bibitem[GL]{GL}
C. Gutierrez \& E. Lanconelli, \emph{Schauder estimates for sub-elliptic equations}, J. Evol. Equ. \textbf{9}~ (2009), no. 4, 707-726.



\bibitem[H]{H} L. H\"ormander, \emph{Hypoelliptic second-order differential equations}, Acta Math., 119 (1967), 147-171.

\bibitem[Je1]{Je1}
D. S. Jerison, \emph{The Dirichlet problem for the Kohn Laplacian on the Heisenberg group. I}, J. Funct. Anal. \textbf{43}~ (1981), no. 1, 97-142.


\bibitem[Je2]{Je2}
D. S. Jerison, \emph{The Dirichlet problem for the Kohn Laplacian on the Heisenberg group. II}, J. Funct. Anal. \textbf{43}~(1981), no. 2, 224-257.

\bibitem[Je3]{Je3}
D. S. Jerison,  \emph{Boundary regularity in the Dirichlet problem for $\Box_b$ on CR manifolds}, Comm. Pure Appl. Math. \textbf{36}~(1983), no. 2, 143-181. 



\bibitem[KN]{KN}
J. Kohn \& L. Nirenberg, \emph{Degenerate elliptic-parabolic equations of second order}, Comm. Pure Appl. Math, \textbf{20}~(1967), 797-872. 

\bibitem[LU]{LU}
  E. Lanconelli \& F. Uguzzoni, \emph{On the Poisson kernel for the Kohn Laplacian},
Rend. Mat. Appl. (7) \textbf{17}~(1997), no. 4, 659--677.

\bibitem[LM]{LM}
J.-L. Lions \& E. Magenes, \emph{Non-homogeneous boundary value problems and applications}, Vol. I. Translated from the French by P. Kenneth. Die Grundlehren der mathematischen Wissenschaften, Band 181. Springer-Verlag, New York-Heidelberg, 1972. xvi+357 pp.



\bibitem[NSW]{NSW}
A. Nagel, E. Stein \& S. Wainger, \emph{Balls and metrics defined by vector fields. I. Basic properties}, Acta Math. \textbf{155}~ (1985), 103-147.


\bibitem[RS]{RS}
L. P. Rothschild \& E. M. Stein,
\emph{Hypoelliptic differential operators and nilpotent groups}, 
Acta Math. \textbf{137}~(1976), no. 3-4, 247-320. 

\bibitem[S]{Snice}
E. M. Stein, \emph{Some problems in harmonic analysis suggested by symmetric spaces and semi-simple groups}, Actes du Congr\`es International des Math\'ematiciens (Nice, 1970), Tome 1, pp. 173-189. Gauthier-Villars, Paris, 1971.


\bibitem[X2]{Xu} C. J. Xu, \emph{Regularity for quasilinear second order subelliptic equations}, Comm. Pure Appl. Math., (1992), 77-96.

\bibitem[V]{V} 
V. S. Varadarajan, \emph{Lie Groups, Lie Algebras, and Their Representations}, Springer-Verlag,
New York, Berlin, Heidelberg, Tokyo, 1974.



\end{thebibliography}
\end{document}